\documentclass[a4paper,12pt]{amsart}

\usepackage{amsmath}
\usepackage{amssymb,amsbsy,amsmath,amsfonts,amssymb,amscd}
\usepackage{latexsym}
\usepackage{amsthm}
\usepackage{graphics}
\usepackage{color}
\usepackage{tikz}
 \usepackage{tikz-3dplot}
\usetikzlibrary{arrows}
\usepackage{tikz-cd}
\usepackage{pdfpages} 
\usepackage{longtable}
\usepackage{xy}
\usepackage{array}
\usepackage{color}
\usepackage{comment}

\input xy
\xyoption{all}

\newcommand\sA{{\mathcal A}}

\newcommand\sI{{\mathcal I}}

\newcommand\sL{{\mathcal L}}

\newcommand\sB{{\mathcal B}}
\newcommand\sP{{\mathcal P}}

\newcommand\sH{{\mathcal H}}

\usepackage{textcomp} 
               
\newcommand\la{\lambda}

\newcommand\be{\beta}
\newcommand\e{\epsilon}
\newcommand\s{\sigma}

\newcommand\De{\Delta}
\newcommand\ga{\gamma}
\newcommand\de{\delta}

\DeclareMathOperator{\Pic}{Pic}

\newcommand{\sS}{\ensuremath{\mathcal{S}}}

\newcommand{\NN}{\ensuremath{\mathbb{N}}}
\newcommand{\hol}{\ensuremath{\mathcal{O}}}

\newcommand{\PP}{\ensuremath{\mathbb{P}}}

\newcommand{\FF}{\ensuremath{\mathbb{F}}}

\newcommand{\ra}{\ensuremath{\rightarrow}}

\def\eea{\end{eqnarray*}}
\def\bea{\begin{eqnarray*}}

\DeclareMathOperator{\ord}{ord}

\newcommand\dual{\mathrel{\raise3pt\hbox{$\underline{\mathrm{\thinspace d
\thinspace}}$}}}
\newcommand\qe{\ifhmode\unskip\nobreak\fi\quad $\Box$}       

\def\BOX{\hfill\lower.5\baselineskip\hbox{$\Box$}}

\newtheorem{theorem}{Theorem}
\newtheorem{theo}[theorem]{Theorem}
\newtheorem{rem}[theorem]{Remark}

\newtheorem{prop}[theorem]{Proposition}
\newtheorem{cor}[theorem]{Corollary}
\newtheorem{lemma}[theorem]{Lemma}
\newtheorem{example}[theorem]{Example}

\newtheorem{claim}[theorem]{Claim}

\newtheorem{main-claim}[theorem]{Main Claim}

\theoremstyle{definition}
\newtheorem{defin}[theorem]{Definition}

\newenvironment{dedication}
        {\begin{quotation}\begin{center}\begin{em}}
        {\par\end{em}\end{center}\end{quotation}}

\usepackage{hyperref}

\makeatletter
\def\tagform@#1{\maketag@@@{\ignorespaces#1\unskip\@@italiccorr}}
\makeatother

\newcolumntype{H}{@{}>{\lrbox0}l<{\endlrbox}} 

\begin{document}

\title[Singularities of quartic surfaces]{Singularities  of  normal quartic  surfaces II (char=2)}
\author{Fabrizio Catanese}
\address{Lehrstuhl Mathematik VIII, 
 Mathematisches Institut der Universit\"{a}t
Bayreuth, NW II\\ Universit\"{a}tsstr. 30,
95447 Bayreuth, Germany \\ and Korea Institute for Advanced Study, Hoegiro 87, Seoul, 
133--722.}
\email{Fabrizio.Catanese@uni-bayreuth.de}

\author{Matthias Sch\"utt}
\address{Institut f\"ur Algebraische Geometrie, Leibniz Universit\"at
  Hannover, Welfengarten 1, 30167 Hannover, Germany\\ and\;\;\;\;\;\;\;\;\;\;\;\;\;\;\;\;\;\;\;\;\;\;\;\;\;\;\;\;\;\;\;\;\;\;\;\;\;
  \linebreak
  Riemann Center for Geometry and Physics, Leibniz Universit\"at
  Hannover, Appelstrasse 2, 30167 Hannover, Germany}

\email{schuett@math.uni-hannover.de}
\date{\today}

\thanks{
The first author acknowledges support of the ERC 2013 Advanced Research Grant - 340258 - TADMICAMT}

\maketitle

\begin{dedication}
Dedicated to Herb Clemens on the occasion of his 82-nd  
 birthday.
\end{dedication}

\begin{abstract}
We show, in this second part,  that the maximal number of singular
points of a  normal quartic surface $X \subset \PP^3_K$ defined over an algebraically closed field $K$ of characteristic $2$ is at most   $14$,
 and that, if we have $14$ singularities, these are nodes {and moreover the
 minimal resolution of $X$ is a supersingular K3 surface.}

We produce an irreducible  component, of dimension $24$, of the variety of
 quartics with $14$ nodes.
 
We also exhibit  easy examples of  quartics with $7$ $A_3$-singularities.
\end{abstract}

\tableofcontents
 
\setcounter{section}{0}

\section{Introduction}

Once upon a time\footnote{It was in August-September 1976} in Cortona there was a Summer school with wonderful courses held by Herb Clemens and Boris Moishezon. 
 The first author
had the privilege of attending the Summer school. On that occasion Herb lectured on several beautiful classical topics,
and these lectures formed the basis of a lovely  book \cite{clemens}. Even if the course and the book were devoted to complex curves, yet characteristic $p$ appeared on the stage, and was used by Clemens to explain the `Unity of Mathematics' (section 2.12). In this spirit we are happy to dedicate this 
`characteristic $2$' paper to Herb.

\smallskip

These are our main results. 
 They feature  the property of the minimal resolution $S$ of  
 being a supersingular K3 surface (i.e.\ with Picard number $\rho=22$).\footnote{For the reader who has never seen  
 such a surface, an easy example is provided in 
 Corollary \ref{A3} in Section \ref{ss:plane}, as the resolution of a quartic surface with $7$ $A_3$ singularities, providing
 $21$ independent $(-2)$-curves on $S$ which together with the hyperplane section $H$ generate a rank $22$ 
 finite index sublattice
 of $\Pic(S)$.}
The following is our main theorem:

\begin{theo}\label{maintheo}
\label{theo}
A normal quartic surface $X \subset \PP^3_K$ defined over an algebraically closed field $K$ of characteristic $2$ 
contains no more than 14 singular points. 
If the maximum number of 14 singularities is attained,
then all singularities are nodes
and the minimal resolution is a supersingular K3 surface.
The variety of quartics with $14$ nodes contains 
an irreducible  component, of dimension $24$.

\end{theo}

 If the minimal resolution $S$ of a normal quartic $X$  is  
 not a supersingular K3 surface,  then
Theorem \ref{theo} shows that 
 $X$ has at most  $13$ singular points.
 This bound
 is not sharp; 
 we have examples with $12$ nodes,
 and we will show in a forthcoming paper (part III) that if $S$  is a K3 surface which is 
 not supersingular, then $X$ has at most  $12$ singular points.

The proof uses mostly classical techniques, notably the Gauss map,
but there are some ingredients (notably the main claim in Section \ref{s:proof})
which build on the theory of genus one fibrations (see Section \ref{s:g=1}).

We emphasize that each ingredient has some special feature in characteristic $2$;
for instance, the Gauss map of a normal surface in $\PP^3$ need not be birational,
and double points behave differently (this affects the degree formula, 
 see Section \ref{s:Gauss}).
The dual surface can be a plane, as we study in Section \ref{ss:plane}.
Elliptic fibrations feature wild ramification (at certain additive fibres),
which has surprising consequences for supersingularity (see Section \ref{s:g=1}).
The notion of genus one fibration also encompasses quasi-elliptic fibrations
whose properties we exploit in Section \ref{s:quasi},
especially with a view to the dual surface.

Naturally Theorem \ref{theo} leads to the question about what is true for other quasi-polarized K3 surfaces
in characteristic $2$,  which we plan to  address in part III as well.

\emph{Convention:}
We work over an algebraically closed field $K$,
mostly of characteristic 2, though many results
may also be stated over non-closed fields.

\section{The Gauss map}
\label{s:Gauss}

We consider in this section a 
normal  quartic
surface $$ X = 
 \{F(x)=0\}
 \subset \PP^3$$  and summarize and extend some considerations made in 
the first part, \cite{cat21} 
 in order to gain control over the (number of) singular points of $X$.
 
 The Gauss map $\ga : X \dasharrow \sP : = (\PP^3)^{\vee}$ is the rational map given by 
$$ \ga(x) : = \nabla F (x), \;\; x \in X^0 : = X \setminus \mathrm{Sing}(X).$$

We let $X^{\vee} : = \overline{\ga(X^0)}$ be the closure of the image of the Gauss map,  
which is a morphism on $X^0$, and becomes a morphism $\tilde{\ga}$ 
on a suitable blow up $\tilde{S}$ of the minimal resolution $S$ of $X$. $X^{\vee}$ is called the dual variety of $X$.

In order to compute the degree of $X^{\vee}$ (this is defined to be  equal to zero if $X^{\vee}$ is a curve), we consider a line $\Lambda \subset \sP$
such that  $\Lambda$ is  
transversal to the map $\tilde{\ga}$, this means: 

\begin{enumerate}
\item[1)] $\Lambda \cap X^{\vee} = \emptyset $ if $X^{\vee}$ is a curve; 
\item[2)] $\Lambda $ is not tangent to $X^{\vee}$ at any smooth point, and neither contains   any singular point of $X^{\vee}$,
nor any point $y$ where the dimension of the fibre $\tilde{\ga}^{-1} (y)$ equals $1$, 
so that 
\item[3)]
 $\Lambda \cap X^{\vee}  $ is in particular a subscheme consisting  of $\deg(X^{\vee})$ distinct points,
and its inverse image in $\tilde{S}$ is a finite set.
\end{enumerate}

By a suitable choice of the coordinates, we may assume that 
$$\ga^{-1} ( \Lambda ) \subset   X \cap \{F_1 = F_2 = 0\}\;\;\;\;\;\; 
(F_i = \partial F/\partial x_i).
$$

The latter  is a finite set, hence by Bezout's theorem it consists of $4 \cdot 3^2= 36$  points counted with multiplicity,
including the singular points of $X$. 

We  have therefore proven the following (probably well known)  formula:

$$ (DEGREE -  FORMULA) \ \  \deg(\ga) \deg(X^{\vee} ) = 36 - \sum_{P \in \mathrm{Sing}(X)} (F,F_1, F_2)_P,$$
where the symbol $(F,F_1, F_2)_P$ denotes the local intersection multiplicity at $P$,
defined by
$$ (F,F_1, F_2)_P: =  \dim_K ( \hol_{\PP^3,P} / (F,  F_1, F_2)) = \dim_K ( \hol_{X,P} / ( F_1, F_2))  .$$ 

Under the above assumptions this intersection multiplicity is zero unless $P$ is a singular point, and then we have
$$ (F,F_1, F_2)_P \geq 2  \;\; \forall P \in \mathrm{Sing}(X) .$$
 The integer $(F,F_1, F_2)_P$ shall be called the Gaussian defect.

\subsection{Calculation of Gaussian defects}

Since quartics with triple points were treated in \cite{cat21},
 we will mostly be concerned with double points,
but we will cover all types of singularities in Proposition \ref{9}.
Double point singularities
 are divided into three rough types
according to the rank of the tangent quadric at $P$:
\begin{enumerate}
\item
nodes: here the quadric is smooth, and we have an $A_1$-singularity, formally equivalent to $xy = z^2$;
 the nodes give a contribution $(F,F_1, F_2)_P=2$ to the Gaussian defect;
\item
biplanar double points: here the quadric consists of two planes, and we have an $A_n$-singularity with $n\geq 2$, 
formally equivalent 
to $xy = z^{n+1}$ (see for instance \cite{cat21a});  the biplanar double points of type $A_{n}$
give a contribution of $n+1$ to the Gaussian defect;
\item
uniplanar double points: here the quadric consists of a double  plane, and we have several types (see \cite{artin}, \cite{artin2},
\cite{roczen}), the Taylor 
development 
 is of the form $ x^2 + \psi =0$, where $\psi$ has order $\geq 3$; the uniplanar 
 double points give a contribution of order at least $8$ to the Gaussian defect, since $ (F,F_1, F_2)_P = (F, \psi_1, \psi_2) \geq 8$.

\end{enumerate}

\begin{prop}\label{gaussestimate}
Let $X$ be  a normal quartic surface in $\PP^3$.

(I) If  $X$ has  $\nu$ singular points of multiplicity $2$,  
among them $b$ biplanar 
double points, and $u$ uniplanar double points, then:
\begin{eqnarray}
\label{eq:deg}
36 - \deg(\ga) \deg(X^{\vee} ) \geq    2 \nu + b + 6 u .
\end{eqnarray}
 (II) 
 If $X$ contains  a node,
then the exceptional curve $E$ in $S$ resolving the node maps to a line via an inseparable map of degree
 two. In particular the Gauss map cannot be birational if $X^{\vee} $ is a normal surface.
 
 (III)  The dual variety $X^{\vee}$ cannot be a line. 
 
 (IV) For $\nu \geq 13$, the dual variety $X^{\vee}$ is  an irreducible surface; 
 in particular $\deg(\ga) \deg(X^{\vee} ) \geq 2$, and,  if  $\deg(\ga) =1$, $X^\vee$  is  non-normal and $ \deg(X^{\vee} ) \geq 3$.

 (V)  For $\nu \geq 14$, if the dual variety $X^{\vee}$  is not a plane, then the singularities
of $X$ are all of type $A_n$ ($u=0$).
 
\end{prop}
\begin{proof}
(I): follows since the nodes give a contribution equal to $2$ to the Gaussian defect, the biplanar double points of type $A_{n}$
give a contribution  $n+1 \geq 3$, the uniplanar 
 double points give a contribution  at least $8$.

(II): given a node $P$, an $A_1$-singularity, then
the affine Taylor development at  $P$ is given by
$$F = xy + z^2 + \psi (x,y,z) =0$$
and the Gauss map on the exceptional curve $ E$,
 given as a conic $E = \{ xy + z^2=0\} \subset \PP^2$,
is  given by $(x,y,0,0)$. If  $X^{\vee} $ is a normal surface and $\ga$ is birational onto its image, then 
$$\tilde{\ga} : \tilde{S} \ra X^{\vee} $$ is an isomorphism over the complement of a finite number of points of
$X^{\vee} $, a contradiction since $E$ maps $2$ to $1$ to a line.

(III):  if $X^{\vee}$ is a line, then there are projective coordinates in $\PP^3$ such that
 the partial derivatives with respect to 2 variables, say $z,w$, are identically zero; hence
$$ X = \{ a z^4 + b w^4 + c z^2 w^2 + z^2 D(x,y) + w^2  E(x,y) + f (x,y) = 0 \}.$$

Writing $$ D(x,y) = d_1 x^2 + d_2 y^2 + d xy,  E(x,y) = e_1 x^2 + e_2 y^2 + e xy,$$
$$ f(x,y) = q(x,y)^2 + f_1 x^3 y + f_2 x y^3,$$
we see that 
$$ \mathrm{Sing}(X) = X \cap \{  y M = x M = 0\}, \ M : = d z^2 + e w^2 +  f_1 x^2 + f_2  y^2,$$
hence Sing$(X) \supset X \cap \{ M=0\}$ and $X$ is not normal.

(IV):  since for $\nu \geq 13$ there must be a node by the degree formula,  $X^{\vee}$ 
contains a line; but $X^\vee$ cannot be a line by (III), 
hence  it is a surface; the rest follows from (II) and from the fact that an irreducible quadric is normal.

 (V)  For $\nu \geq 14$, if the dual variety $X^{\vee}$  is not a plane, then  $ \deg (\ga) \geq 2$, or
$ \deg (X^{\vee}) \geq 3$  (since if $\ga$ is birational, then $X^{\vee}$ is not normal  by (IV)),
hence $ 2 \nu + b + 6 u 
\leq 33$, hence $u=0$ and   the singularities
of $X$ are all of type $A_n$.

\end{proof}

\begin{rem}
\label{rem:non-RDP}

The degree formula can be improved substantially by taking the precise types
of singularities into account as the proof of (I) shows.
For instance, the biplanar double points contribute $b$ in \ref{eq:deg} exactly
when they all have type $A_2$. 
\end{rem}

\subsection{Gaussian defect of non-rational double point singularities}

In the spirit of Remark \ref{rem:non-RDP}, 
we take a closer look at those singularities which are not rational double point.
This will enable us to strengthen the results of Proposition \ref{gaussestimate},
and to decide when the minimal resolution $S$ of $X$ is a K3 surface (see Proposition \ref{cor:9}).

\begin{prop}\label{9}
If a given singularity $P$ on $X$ is not a rational double point, we have, for 
a general (hence any)
 choice of the 
affine local coordinates at $P$,
\begin{eqnarray}
\label{eq:>=10}
(F,F_1, F_2)_P  \geq 10.
\end{eqnarray}
\end{prop}

\begin{proof}

First of all, for a triple point the Gaussian defect is at least 12, 
 since the Gaussian defect is greater or equal to the product of the respective 
orders of $F, F_1, F_2$ at $P$, and a double point which is not a rational double point
must be  a uniplanar double point.

We can therefore assume that the  affine Taylor development   of $F$ at the point $P$ is of the form 
$$ \ F (x_1, x_2,x_3) = x_1 ^2   +   G (x) + B (x),$$
where $G$ is homogeneous of degree $3$ and $B$ of degree $4$.

We may take local coordinates such that $ x: = x_1,$ and where $y,z$ are generic linear forms vanishing at $P$,
hence the Gaussian defect will be the intersection number 
$ (F, F_y, F_z)$ at the point $P$.

This said, we can write 
$$ F (x,y,z) = x^2 ( 1 + A(x,y,z)) + x g(y,z) + g' (y,z) + B' (x,y,z),$$
and multiplying by $ ( 1 + A(x,y,z))^{-1}$ we get  a formal power series  equation
\begin{eqnarray}
\label{eq:f} f =  x^2 + x g(y,z) + g' (y,z) + b (x,y,z),
\end{eqnarray}
where $g$ is a quadratic form, $g'$ is a cubic form, and  the power series $b$ has order at least 4.

We  consider the blow up of the singular point $P \in X$. 

The equation of $X$ is $ f(x,y,z) = x^2 + x g(y,z) + g'(y,z) + b(x,y,z)  = 0$;
set now :
$$ x = t\, \xi, \; \ y = t\, \eta, \; \ z= t\, \zeta$$
(here $t=0$ is the equation of the exceptional divisor, isomorphic to $\PP^2$ and $(\xi,  \eta,  \zeta)$
are homogeneous coordinates in $\PP^2$)
so that the equation of the blow up is 

$$ \xi^2 + t \xi \ g (\eta,  \zeta) + t  \ g' (\eta,  \zeta) + t^2 \ b (\xi, \eta,  \zeta)=0.$$

On the exceptional line $\{ t = \xi =0 \}$ the singular points are the roots of $ g' $.

Hence either $ g' $ is identically zero, or the blow up is normal.

If $ g' $ does not have a multiple root,
 we get 3 nodes, hence   $P$ is a singularity of type $D_4$ and we are done.
 
Therefore  we may assume that $g'$ has a multiple root (or is identically zero) and apply a linear transformation such that 
$ y =0$  is this root.

We want to show that the length of the Artin algebra 
$$\sA : = \hol_P / ( f, f_y, f_z) $$
is $\geq 10$.

A fortiori it will suffice to replace the algebra $\sA$ 
by the quotient algebra 
$$\sA_4 : = \hol_P / (( f, f_y, f_z)  + \mathfrak M_P^4)$$ 
or 
by the quotient algebra 
$$\sA_5 : = \hol_P / ((f, f_y, f_z)  + \mathfrak M_P^5).$$
Inside the algebra $\sB_4 : =  \hol_P /  \mathfrak M_P^4$, the ideal $\sI$ generated by $ f, f_y, f_z$
is generated as a vector space by the vectors 
$$f, f_y, f_z, x f, x f_y, x f_z, y f, y f_y, y f_z, z f, z f_y, zf_z $$
where the first three have order at least $2$, and the latter at least $3$.

Since $\sI \subset \mathfrak M_P^2$,  which has colength $4$, it suffices to show that 

\begin{enumerate}
\item[1)]
 $(\sI + \mathfrak M_P^3)/ \mathfrak M_P^3$ has codimension at least $3$  in $\mathfrak M_P^2/ \mathfrak M_P^3$, which is a
 6-dimensional vector space.
\item[2)] 
$((\sI \cap  \mathfrak M_P^3) +  \mathfrak M_P^4)  / \mathfrak M_P^4$ has codimension at least $2$  in $\mathfrak M_P^3/ \mathfrak M_P^4$,
which is a 10-dimensional vector space.
\item[3)] if in 2) codimension $2$ occurs, then 
$((\sI \cap  \mathfrak M_P^4) +  \mathfrak M_P^5)/ \mathfrak M_P^5$ has codimension at least $2$  in $\mathfrak M_P^4/ \mathfrak M_P^5$.
\end{enumerate}

We can now write 
$$ f = x^2 + x g(y,z) + cy^3+dy^2z  \;\; (\mathrm{mod} \  \mathfrak M_P^4),$$
$$ \  f_y = a xz + c y^2, \;\;\; \ f_z = a xy + d y^2  \;\; (\mathrm{mod} \  \mathfrak M_P^3),  $$
where  $c,d$ may be equal to zero.

The first assertion is clear, since modulo $\mathfrak M_P^3$ we just have three vectors,  
$x^2, a xz + c y^2, a xy + d y^2$. 
In fact, if $a=0$, then the vectors are linearly dependent modulo $\mathfrak M_P^3$,
so $(\sI + \mathfrak M_P^3)/ \mathfrak M_P^3$ has codimension at least $4$  in $\mathfrak M_P^2/ \mathfrak M_P^3$.
As one easily checks that $((\sI \cap  \mathfrak M_P^3) +  \mathfrak M_P^4)  / \mathfrak M_P^4$  has codimension at least $3$  in $\mathfrak M_P^3/ \mathfrak M_P^4$ in this situation,
it follows that $(F,F_1, F_2)_P  \geq 11$ as desired.
Hence, in what follows, we will assume $a\neq 0$ and thus normalize $a=1$.


For the second assertion, it  suffices to show that, modulo $\mathfrak M_P^4$, we get 
 a subspace in degree 3 of dimension at most $8$.
 
 From $f$, in degree 3 we get 
$ x^3, x^2 y, x^2 z,$
and modulo the   subspace generated by the above vectors 
we get  
$xf_y\equiv cxy^2, xf_z\equiv dxy^2$: these vectors are all contained 
in  the 4-dimensional subspace $V$ generated by 
$x^3, x^2 y, x^2 z, x y^2.$ 
Since there are only 4 further generators in degree 3 (given below), this already proves the second assertion.

For future use, we further investigate $(\sI + \mathfrak M_P^4)  / \mathfrak M_P^4$. 
Modulo $V$, we  get the 4 vectors, 
$$ yf_y = y (x z + c y^2),  \; zf_y = z (x z + c y^2), \;
yf_z \equiv d y^3, \;  zf_z=z (x y  + d y^2).$$
These are linearly independent if and only if $d\neq 0$;
in that case, their span is  generated by
$xyz, xz^2, y^3, y^2z$, in agreement with the second assertion.

Note that, if $d=0$, then $(\sI \cap  \mathfrak M_P^3) / \mathfrak M_P^4$ has codimension at least $3$  in $\mathfrak M_P^3/ \mathfrak M_P^4$, and the main claim   of our proposition follows readily.

Hence we will assume $d\neq 0$ in what follows.

In order to prove the third assertion 
we observe that the ideal $$
\mathcal I':=(x^2, x y^2, x z^2, xyz, y^3, y^2z)$$ 
arising from the monomials in the above computations contains all monomials
of degree 4 except for $z^4$ and $z^3y$.

Define $W$ to be the subspace of  $\mathfrak M_P^4/ \mathfrak M_P^5$ generated by $\mathcal I'$,
i.e.\ by the monomials containing $x$
or divisible by $y^2$.

Working in 
$$ U: = (\mathfrak M_P^4/ \mathfrak M_P^5)/W \cong K z^4 \oplus K z^3y
$$
we want to show that $((\sI \cap  \mathfrak M_P^4) +\mathfrak M_P^5) / \mathfrak M_P^5$
maps to zero in $U$.

Observe that, in degree $3$, $b_y \equiv \la y^2 z + \mu z^3, b_z \equiv \la y^3 + \mu y z^2 $
modulo $\mathcal I'$.

But in order to get this, we must have some  degree 1  relation between the quadratic parts of $f, f_y, f_z$,
giving then rise in degree 4 
to some non-zero vector in $U$.

In fact, one relation is obvious, namely
\[
(  cy + dz)f+ dxf_y+ cxf_z\in \mathfrak M_P^4,
\]
but it only produces  
multiples of $x$ and $y^2$ in degree 4, i.e.\ zero in $U$.

Direct calculation shows that the above is the only relation in degree 1 occurring, so we are done.

\end{proof}

We will use the proposition soon to derive a criterion for  the minimal resolution $S$ of $X$ to be  a K3 surface
(Proposition \ref{cor:9}),
but as a preparation 
we have to discuss the case where the dual surface $X^\vee$ is a plane.

\section{When the dual surface is a plane}
\label{ss:plane}

 We continue to consider a normal quartic surface $X\subset \PP^3$.
The term $\deg(\gamma)\deg(X^\vee)$ in \eqref{eq:deg}
deems it essential to study the case where the dual surface $X^{\vee}$ is a plane.

In this case there are coordinates $(x_1, x_2, x_3, z)$ such that   the partial derivative of $F$ with respect to $z$ is identically zero,
hence
\begin{eqnarray}
\label{eq:plane}
 X = \{ (x,z) | a z^4 + z^2 Q(x) + B(x)=0\}.
 \end{eqnarray}
We are going to show that for the general such surface $X$ has $14$ nodes as singularities
 (Theorem \ref{dual=plane}). 
This will prove part of Theorem \ref{theo}.

There are a few special equations where $X^\vee$ is a plane
which require extra treatment.
Two of them were contained in part one of this paper \cite{cat21}:

\begin{enumerate}
\item[(1)]
 $ X = \{ (x,z) \mid  z^2 Q(x) + B(x)=0\}$ ($a=0$)  is the case where there is a singular point $P$ ($x=0$)
such that projection with centre $P$ is inseparable.
\item[(2)]
 $Q(x)$ is the square of a linear form (see Proposition \ref{special}). 
\end{enumerate}
 One more equation will appear in Proposition  \ref{insep}.
Together with Lemma \ref{dp}, they will suffice to prove the instrumental fact that, with at least 13 singular points,
the minimal resolution $S$ of $X$ is a K3 surface (Proposition \ref{cor:9}).

\medskip

The curve $\{ B(x)=0\}$  obtained from equation  \ref{eq:plane}  is a plane quartic curve, and we want now to establish some simple properties of 
plane quartic curves which will be relevant for our issues.

\subsection{The  strange points of a plane curve of even degree in characteristic $=2$}
We define here, as in \cite{cat21}, the strange points of a plane curve $\{ B(x) = 0\}\subset\PP^2$
to be the points outside the curve where the gradient 
$\nabla B$ vanishes.

We have seen in Part I (\cite{cat21}) for the case of  a general plane quartic curve $\{ B(x) = 0\}\subset\PP^2$:

\begin{prop}\label{7}
 For a homogeneous quartic polynomial $B \in K [x_1,x_2,x_3]_4$
 let $\Sigma$  be the critical locus of $B$, the locus 
where the gradient 
$\nabla B$ vanishes. 
If $\Sigma$ is a finite set, then it consists of at most  7 points.

For $B$ general, $\Sigma$ consists of exactly 7 reduced points.
\end{prop}

The above result does indeed nicely extend to the case of a homogeneous polynomial of even degree $B \in K [x_1,x_2,x_3]_{2m}$.

As noticed in \cite{cat21}, because of the Euler identity $$x_1 B_1 + x_2 B_2 + x_3 B_3 =0$$
among partial derivatives,
taking coordinates such that the line $\{ x_3=0 \}$ does not intersect $\Sigma$,
it follows that $$\Sigma = \{ x \mid  B_1(x)= B_2(x)=0, x_3 \neq 0\}.$$

For a polynomial of the Klein form 
$$ B_0 : = x_1^{2m-1} x_2 + x_2^{2m-1} x_3 + x_3^{2m-1} x_1,$$
the critical scheme is defined by 
$$ x_i^{2m-1} = x_{i+1}^{2m-2} x_{i+2} \;\;  \Longrightarrow  \;\; x = (\e, 1, \e^{2m-1}), \ \e^{(2m-2)(2m-1) + 1} = 1.$$

Letting $s$ be  the general number of strange points ($s: = |\Sigma|$),
we have therefore that, setting $d = 2m$, $s$ lies in the interval
$$ (d-1)(d-2)+ 1= (d-1)^2 - (d-2)   \leq s \leq (d-1)^2 .$$

{
\begin{prop}
The  number of strange points $s: = |\Sigma|$ ($\Sigma : = \{ \nabla B(x)=0\}$) of a  general homogeneous
polynomial $B(x_1, x_2,x_3)$ of even degree $d$
is   equal to $ s = (d-1)(d-2)+ 1 $.

Whenever the 
 subscheme $\Sigma$ 
 is finite,
 its length equals $s$.
\end{prop}
\begin{proof}

 To show this, two steps suffice:

I) 
 if $\Sigma$ is finite, then the scheme  $ \{ x \mid B_1(x)= B_2(x)=0\}$ is finite
 for general choice of coordinates, and $\Sigma \cap \{ x_3=0\}$ is empty in general;

II) the subscheme $ \{ x| B_1(x)= B_2(x)= x_3 =0\}$, if we  write
$$ B = x_3 B' + \be ( x_1, x_2) + q(x_1, x_2)^2, \ \be ( x_1, x_2) = \sum_{n \ \text{odd}} a_n x_1^{n} x_2^{2m-n},$$
equals  the subscheme 
$$ \left\{ x\, \middle\vert \, x_3 =  \sum_{n \ \text{odd}} a_n x_1^{n-1} x_2^{2m-n} =  \sum_{n \ \text{odd}} a_n x_1^n x_2^{2m-n-1} =0 \right\} = $$
$$  \left\{ x \, \middle\vert \, x_3 = x_2 \left( \sum_{n \ \text{odd}} a_n x_1^{n-1} x_2^{2m-n-1}\right) = x_1 \left( \sum_{n \ \text{odd}} a_{n-1} x_1^{n-1} x_2^{2m-n-1}\right) =0 \right\}.$$

I) holds since changing variables we get that the  new partials are general linear combinations of the 
old partials: if $\Sigma$ is finite then we can keep $B_1$ fixed and vary $B_2$ so that it has no common factor with
$B_1$; hence the result holds for general choice of linear coordinates, and  the rest is obvious.

Our result follows then  from I) and II), since then the scheme $\Sigma$ is disjoint from the length $(d-2)$ scheme  
$$\left\{ x_3=  \sum_{n \ \text{odd}} a_{n-1} x_1^{n-1} x_2^{2m-n-1} =0 \right\},
$$
 and we conclude
since their union is the complete intersection subscheme  $ \{ x| B_1(x)= B_2(x)=0\}$, which has length $(d-1)^2$.

\end{proof}
}
\begin{rem}
 A  more general result (also valid in other characteristics) is  contained in Theorem 2.4 of \cite{liedtkecan},
whose formulation,  however, does neither  mention derivatives nor critical sets.
\end{rem}

\subsection{Supersingular quartics with $7$ $A_3$-singularities}

A first  immediate consequence of the previous result is:

\begin{cor}\label{A3}
For $B$ a  homogeneous polynomial $B \in K [x_1,x_2,x_3]_4$, a normal  quartic surface of the form:
$$ X : = \{ (x,z) \mid  z^4 + B(x) = 0 \}$$ has at most $7$ singular points.

If $B$ is  general, $X$  has 7 $A_3$-singularities.
\end{cor}
\begin{proof}
The singular points of $X$ are in bijection with the critical set $\Sigma$ of $B$, which consists of $7$
reduced points  for $B$ general. Hence at these points there are local coordinates $u,v$ such that $ B = a^4 + uv$,
hence the local equation of $X$ is $ (z+a)^4 = uv$, and 
we have an $A_3$-singularity.

\end{proof}

\subsection{Quartics with $14$ nodes blowing up to $7$ lines in the plane under the Gauss map}

A second immediate consequence concerns the quartics with dual surface equal to a plane, and with
 a singular point such that the second order term 
$Q$ of the Taylor development (see formula  \ref{eq:plane}) is  equal to the square of a linear form.

\begin{prop}\label{special}
Consider a normal quartic of equation 
$$ X = \{ (x,z) \mid  z^4 + z^2 x_1^2 + B(x)=0\},$$
where $B$ is a  homogeneous polynomial of degree $4$.
Then $X$ has at most $14$ singular points, inverse image of the (at most $7$) points in the plane where $\nabla B(x)=0$.

For general choice of $B$, $X$ has exactly $14$ nodes as singularities.

The Gauss map $\ga$ is inseparable, it factors through the projection $ (x,z) \mapsto x$ and a degree two map
 $\PP^2 \ra \PP^2$,  $ x \mapsto \nabla B(x)$.
 
 {In particular, if $X$ has $14$ singular points, these are nodes.}

\end{prop}
\begin{proof}
The Gauss map is given by 
$$ \ga(x,z) = (\nabla B(x), 0).$$
The singular points are the inverse image of the critical locus of $B$, $\Sigma = \{ x \mid  \nabla B(x)=0\}$.
$\Sigma$ consists of at most $7$ points by Proposition \ref{7}.

For general $B$ we get $7$ reduced points, and since $z^2$ is the root of a quadratic polynomial with derivative
$x_1^2$, if the line $\{ x_1=0\}$ does not meet the locus $\Sigma $, we get $14$ nodes as singularities.

Observe finally that $ x \mapsto \nabla B(x)$ has degree $2$ since the base locus consists of the  length $7$ subscheme $\Sigma$.

The last assertion follows now easily from the fact that the Gauss map has degree $8$, and 
from the Gauss estimate \ref{eq:deg}
of Proposition \ref{gaussestimate}.

\end{proof}

\subsection{Quartics with inseparable projection from one node}

This is another specialization, corresponding to the case $a=0$ in \ref{eq:plane}, 
but the conic $Q$ is a smooth one.

\begin{prop}\label{insep}
Consider a quartic of equation 
$$ X = \{ (x,z) \mid    z^2 (x_1 x_2 + x_3^2 )+ B(x)=0\},$$
where $B$ is a  homogeneous polynomial of degree $4$.

{
The Gauss map $\ga$ is inseparable, it factors through the degree two projection $ (x,z) \mapsto x$ and a degree four  map
 $\PP^2 \ra \PP^2$.
 
  $X$ has at most $14$ singularities, the node $P = \{x=0\}$, and the inverse image of  a $0$-dimensional subscheme
  of the plane of length $13$. If $X$ has $14$ singularities, these are nodes; and,  for general choice of $B$, $X$ has $14$ nodes as singularities.}

\end{prop}
\begin{proof}
The Gauss map is given by 
$$ \ga(x,z) = (z^2 x_2 + B_1, z^2 x_1 + B_2, B_3).$$
Multiplying by $ Q = (x_1 x_2 + x_3^2 )$ and using the equation of $X$, we get that 
$$ \ga(x,z) = \ga'(x) : = ( B  x_2 + B_1 Q, B  x_1 + B_2 Q, B_3 Q).$$
The base point scheme of $\ga'$   in the plane consists of $\{Q=B=0\}$, which {is a length 8 subscheme which}
 in general consists of $8$ reduced points,
and, since outside of this subscheme  we may assume  that $Q(x)\neq 0$ (since $x_1= x_2 = Q=0$ has no solutions), of the locus 
$$\sS : = \{ B  x_2 + B_1 Q=  B  x_1 + B_2 Q= B_3 =0\}.$$
We observe now that every quartic polynomial can be uniquely written as the sum of a square $q^2$ plus
a polynomial of the special form below 
$$ B' : = \sum_{i\neq j}  b_{ij}x_1^3 x_j + \sum_i c_i x_i x_1x_2x_3.$$
Then, working modulo $(x_3)$, we get:
$$ B' \equiv b_{12} x_1^3 x_2 + b_{21} x_2^3 x_1, \;\;\; B'_2 \equiv b_{12} x_1^3  + b_{21} x_2^2 x_1,$$
$$B'_1 \equiv b_{12} x_1^2 x_2   + b_{21} x_2^3.$$
Hence $ x_1 B_1' \equiv x_2 B_2' \equiv B'  \;(\mathrm{mod} \ x_3)$.

Consider the subscheme
$$\sL : = \{ x_3 = B  x_2 + B_1 Q=  B  x_1 + B_2 Q= 0\}.$$
For $B = B'$ we get $\sL = \{ x_3 =0\}$, because  $ Q = x_1 x_2 + x_3^2$.

  If we now add to $B'$ the square of a quadratic form $q^2$,  
$\sL$ coincides with
$$ \{ x_3 =  q^2 x_2 =  q^2 x_1  = 0\} = \{ x_3 = q^2=0\}.$$
This is a subscheme of length equal to $4$.

Our subscheme $\sS$ is the residual scheme with respect to the above length $4$ scheme $\sL$ of the
scheme 
$$\sH \sB  : = \{ B  x_2 + B_1 Q=  B  x_1 + B_2 Q= x_3 B_3 =0\}.$$
$\sH \sB $  is a Hilbert-Burch Cohen-Macaulay subscheme of codimension $2$, corresponding to the
$2 \times 3$ matrix  with rows
$ ( x_1, x_2, Q)$ and $(B_2, B_1, B)$.

Since the ideal $\sI$ of the subscheme has a resolution 
$$ 0 \ra \hol_{\PP^2} (-6) \oplus  \hol_{\PP^2} (-6) \ra  \hol_{\PP^2} (-5)^2 \oplus  \hol_{\PP^2} (-4) \ra \sI \ra 0,$$
an easy Chern class computation shows that the length of $\sH \sB$ is $17$. Moving $q$, the scheme $\sL$ is disjoint
from $\sS$, hence {we conclude that} the length of $\sS$ is $13$.

The degree of the Gauss map is then $ 2 (25 - 8 - 13)= 8$; hence, by the Gauss estimate \ref{eq:deg} of Proposition
\ref{gaussestimate} we get $ 28 \geq 2 \nu + b + u$,
hence for $\nu =14$ we obtain $14$ nodes. 

That the subscheme $\sS$ consists  in general of $13$ distinct points follows from the examples given in \cite{cat21}, step IV of proposition 3.

\end{proof}
\subsection{A $24$-dimensional family of quartics with $14$ nodes}
We pass now to the general case, where $a\neq 0$, and $Q$ is not a double line.

\begin{theo}\label{dual=plane}
Let $K$ be an algebraically closed field of characteristic $2$, and let $X \subset \PP^3_K$ be a general
quartic hypersurface such that the dual variety is a plane.

Then $X$ has $14$ nodes as singularities  and is unirational, hence supersingular.

These quartic surfaces form an irreducible  component,  of dimension $24$,  of the variety of
quartics with $14$ nodes.
\end{theo}

\begin{proof}
The condition that the dual variety is a plane  is  equivalent to the existence 
 of coordinates $(x,z)$ ($x= (x_1, x_2,x_3)$)
such that
$$ X = \{ (x,z) \mid a z^4 + z^2 Q(x) + B(x)=0\}.$$

Since we have already dealt with the special case $a=0$, and with the cases $Q=0$ or the square of a linear form,
let us assume that $$a=1, \;\; Q(x) = x_1 x_2 + \la x_3^2.$$

The Gauss map is given by
$$ \ga(x,z) = z^2 \nabla Q + \nabla B= ( z^2 x_2 + B_1, z^2 x_1 + B_2, B_3).$$
Hence for the singular points $B_3(x)=0$, which implies $x_1 B_1 + x_2 B_2=0$.
Therefore, for the singular points we have
$$ z^2 = \frac{B_1}{x_2} = \frac{B_2}{x_1}.$$
More precisely, if we have a point $x \in \PP^2$ such that $B_3(x)=0$, and $x_2 \neq 0$,
necessarily we have $ z^2 = \frac{B_1}{x_2} $ and we have a singular point if the equation of $X$ is satisfied,
namely if 
$$  z^4 + z^2 Q(x) + B(x)=0 \;\; \Longleftrightarrow \;\;
B_1^2 + B_1 x_2 Q + B x_2^2=0.$$
An easy calculation shows that
$$B_3 = b_{32} x_3^2 x_2 +  b_{31} x_3^2 x_1+ b_{13} x_1^3  + b_{23} x_2^3 + c_1 x_1^2 x_2 + c_2 x_2^2 x_1,$$
which is in the ideal $(x_1, x_2)$ but does not in general vanish neither on $x_1=x_3=0$ nor on $x_2=x_3=0$.

We look now at the points $x$ where 
$$B_3=B_1^2 + B_1 x_2 Q + B x_2^2= x_2 = 0 \;\; \Longleftrightarrow \;\;
B_3= x_2=B_1= 0 :$$
 these are contained in the set 
$\{x_2=  b_{31} x_3^2 x_1+ b_{13} x_1^3 = 0\}$, which consists of  the point $P'':=\{x_2=x_1=0\}$, and the point
$ P' : = \{x_2=  b_{31} x_3^2 + b_{13} x_1^2 = 0\}$. 

 Since 
$$B_1 = b_{12} x_1^2 x_2 +  b_{13} x_1^2 x_3+ b_{31} x_3^3  + b_{21} x_2^3 + c_3 x_3^2 x_2 + c_2 x_2^2 x_3,$$
 $B_1$ does not in general vanish in $P''$,
but it vanishes indeed in  $P'$.

At the point $P'$, for general choice of $B$, $x_1\neq 0$, hence if $P'$ were to correspond to a singular point of $X$, 
we would have  
$$ z^2 = \frac{B_2}{x_1} \Rightarrow B_2^2 + B_2 x_1 Q + B x_1^2 . $$ 
But the right hand side does not in general vanish at $P'$, since $x_1 \neq 0$, and since we can add to $B$ the
square of a quadratic form $q(x)$ 
 without affecting the partial derivatives.

The number of singular points of $X$ is then equal, by the B\'ezout theorem, to the difference between
$18$ and the intersection multiplicity of $B_3$ and $B_1^2 + B_1 x_2 Q + B x_2^2$ at the point $P'$.

In the special case $ B = x_1^3 x_2 + x_2^3 x_3 + x_3^3 x_1 + q^2$, we get the point  
$P'=\{x_2=x_3=0\}$,
and 
$$B_3 =  x_3^2 x_1  +  x_2^3,\;\;  B_1 =   x_1^2 x_2  +  x_3^3 .$$
The curve $\{ B_3=0\}$ has a cusp with tangent $\{x_2=0\}$, so that $x_3$ has order $3$, $x_2$ has order $2$,
hence $B_1^2 + B_1 x_2 Q + B x_2^2$ has order equal to $4$ for general choice of $q$.

By semicontinuity the intersection multiplicity  is in general at most $4$, hence the `number' of singular
points of $X$ is at least  $14$. But in the  special case of proposition \ref{special}  we have exactly $14$
nodes, so $14$ points counted with multiplicity $1$; hence   by semicontinuity in the other direction
we have in general exactly $14$ nodes.

 That $X$ is unirational, hence supersingular by \cite{shiodass}, follows since $X$ is an inseparable double cover
of the surface 
$$ Y =  \{ (x,w) \mid  a w^2 + w Q(x) + B(x)=0\} \subset \PP(1,1,1,2).$$ 
$Y$ has degree $4$, hence $\omega_Y = \hol_Y (-1)$ and $Y$ is a Del Pezzo surface, hence rational.

The  dimensionality assertion follows by a simple parameter counting, $1 + 6 + 15-1=21$ parameters for the
above polynomial equations, plus $3$ parameters for the plane $X^{\vee}$, which in the chosen equations
is the plane $\{z=0\}$.

{Finally, consider the surface  
$$
X_0 : = \{ (x,z) \mid  z^4 + z^2 l(x)^2  + B_0(x)=0\}, \;\; 
B_0 = x_1^3 x_2 + x_2^3 x_3  + x_3^3 x_1,
$$
and consider the deformations obtained by adding  to the equation of $X_0$ a polynomial 
$$ f : = z \sum_{i=1}^7 \la_i G(i)(x) + z^3 \sum_{j=1}^3 \mu_j L_j(x),$$
where $G(1), \dots, G(7)$ are polynomials of degree $3$ such that $G(i)$
is vanishing at exactly all the critical points
of $B_0$ except the $i$-th point $P_i$, and the linear forms $L_j(x)$ vanish on the points $P_i, 1 \leq i \leq 3, \ i \neq j.$

The polynomial $f$ belongs to a $10$-dimensional vector subspace, and we shall show now that
we get  independent smoothings of ten of the nodes: 
one for each of the pairs of  singular points $P_i',  P_i''$ lying over $P_i$, for $ i=4,5,6,7$,
and two  over each $P_i$ for $i=1,2,3$.

Then, if we choose one  of the two singular points $P_i',  P_i''$ lying over $P_i$, for $ i=4,5,6,7$,
say $P'_i$, 
the map to the local deformation space of the
singularity is of the form (in local coordinates $u, v, \zeta : = (z + z'_i)$ such that $B = uv + {\rm constant}$)

$$uv + (z+z'_i)^2 + \la_i z  + z \sum_{j=1}^3 \mu_j L_j(P_i) (z'_i)^2,$$
since $z^3 = (\zeta + z_i')^3 \equiv z (z'_i)^2 \ (\mathrm{mod} \ \zeta^2)$;
whereas for $j=1,2,3$ the map is given by
$$uv + (z+z'_j)^2 + \la_j z  + z \mu_j  (z'_j)^2,$$
respectively by
$$uv + (z+z''_j)^2 + \la_j z  + z \mu_j  (z''_j)^2.$$

Observe that, if we have a node of equation $ uv + \zeta^2=0$, the local deformations are of the form 
$$ uv + \zeta^2 + c_0 + c_1 \zeta=0,$$
and we have a smoothing iff $c_1 \neq 0$.

It is easily seen that  the deformation yields ten  independent smoothings of the ten nodes
$P_1', \dots, P_7', P_1'', P_2'', P_3''$, hence  it follows that the variety 
of quartics with $14$ nodes, at the point $X_0$, has Zariski tangent space of codimension at least $10$ in the space of all quartics. 
Since the space of all quartics has dimension $34$, and our family is irreducible of codimension $10$,
it follows that at the point $X_0$ our family coincides with the variety of quartics with $14$ nodes, and our family is a component of
this variety.
}

\end{proof}

\begin{rem}
Since our family yields a   dimension 9 locus in the moduli space, we have found an irreducible component of the moduli space
of supersingular K3 surface with a quasi-polarization of degree $4$.
This may be compared to Shimada's results on double sextics
where there is an irreducible component with 21 nodes \cite{Shimada}.
\end{rem}

\subsection{When  the  minimal resolution is a K3 surface}

Concerning the degree of the Gauss map, which is in the above situation generally equal to 8, 
we have a weaker result, which is sufficient, as we will see in Proposition \ref{cor:9},
 for the purpose of showing that the minimal resolution of $X$ is always a K3 surface 
 if the number of singular points is at least 13.
Equivalently, all singularities are rational double points.
 
 \begin{lemma}\label{dp}
 Assume that the normal quartic $X$ has the following equation
 $$ X = \{ (x,z) \; |  \; z^4 + z^2 Q(x) + B(x)=0\},$$
 where the quadratic form  $Q$ is not the square of a linear form.
 
 Then the degree of the Gauss map is at least 4  or $X$ has at most 12 singular points.
 \end{lemma}

\begin{proof}
We use again the normal form where $ Q(x) = x_1 x_2 + \la x_3^2, \ \la \in \{0,1\}$.

 The Gauss map factors through the inseparable double cover (setting $w:= z^2$) of the Del Pezzo surface $Y$ of degree $2$
in $\PP(1,1,1,2)$, such that $  \omega_Y = \hol_Y(-1)$.

The  projection to the $\PP^2$  with coordinates $x$  and the Gauss map to the plane with coordinates $y$
 induce a birational embedding of $Y$   in $\PP^2 \times \PP^2$, since $y = \ga (x,w) = ( w x_2 + B_1, w x_1 + B_2, B_3)$,
hence
$$ y_1/y_3 = (w x_2 + B_1)/B_3 \Rightarrow 
w  = (B_3/x_2 ) ( y_1/y_3 + B_1/B_3).$$

The image lands, as it is immediate to verify, in the flag manifold $\FF$, a smooth divisor of bitype (1,1)

$$ \FF = \left\{ (x,y) \,\middle\vert \, \sum_i x_i y_i =0\right\},$$

and inside $\FF$ the image $Z$  of $Y$ is a divisor of bitype $(d,2)$ where $2d$ is the degree of the Gauss map.

We want to show that $ d >1$.

By adjunction the dualizing sheaf $\omega_Z$ of $Z$ is a divisor of bitype $(d-2, 0)$.
whereas  the canonical system of $Y$ corresponds to a divisor of bitype (-1, 0).
The crucial observation is that, if $d=1$, then these two divisors coincide.

$Y$  has a rational map to $Z$  and composing with  the first projection we get  a morphism, while
composing with the second projection we get the blow up of some  points.

Let  $Y'$ be the blow up of $Y$, such that $\pi : Y' \ra  Z$ is a birational morphism.
Also the second projection $ p : Z \ra \PP^2$ is a birational morphism, moreover the fibres of $p$ are contained in the fibres of 
$ p : \FF \ra \PP^2$, which are isomorphic to $\PP^1$. We blow up the points of $\PP^2$ where the fibre of
 $ p : Z \ra \PP^2$ has dimension 1, obtaining $Z'$.  Then we get a factorization $ Z \ra Z' \ra \PP^2$.

Since  $Z'$ is smooth, and $ Z \ra Z'$ is finite and birational, follows that $ Z \cong Z'$ and $Z$ is smooth.

Now $Z$ and $Y$ are birational  normal Del Pezzo surfaces, and for both the anticanonical divisor is the pull back
of $\hol_{\PP^2}(1)$ under the first projection (to the $\PP^2$ with coordinates $(x)$).

 The  first projection $\phi : Z \ra \PP^2$ has degree two and  is either finite, or its fibres are isomorphic to $\PP^1$.
By normality we have a birational morphism $ \psi : Z \ra Y$. In the first case $\psi$ is an isomorphism,
in the second case it is a minimal resolution of singularites. And since the fibres are smooth rational curves with normal bundle of degree $-2$, then the corresponding singularities of $Y$ are nodes.

This shows that $d=1$ is only possible if  there are  no singular points of $X$ which do not map to singular points of $Y$, and the
latter are nodes.

Since the  singularities of $Y$ correspond to the singularities of $X$ for which $ Q(x)=0$, we see that all the singular points of $X$
satisfy $ Q(x)=0$. Since the singular points of $Y$ are defined by $ Q(x)=0$ and by 3 equations of degree
$3$, it follows that there is a linear combination $B'(w,x) $ of these 3 equations such that 
the singular points of $Y$ are contained in the finite set defined by $ Q(x)= B'(w,x)=0$.

Since $\hol_Y(1)$ has self-intersection equal to  2, $Y$ has at most 12 singular points.

\end{proof}

\begin{prop}\label{13}
\label{cor:9}
If $2 \nu  > 28 - \deg (\ga) \deg (X^{\vee}) $, 
then all singularities of $X$ are rational double points,
and the minimal resolution $S$ is a K3 surface.
In particular, this holds  for  $\nu\geq 13$.
\end{prop}

\begin{proof}
The first statement follows directly from combining Propositions \ref{gaussestimate} and \ref{9}.


Let's deal with the second assertion.

If $X^\vee$ is not a plane, then, by Proposition \ref{gaussestimate} (IV),  $\deg (\ga) \deg (X^{\vee}) \geq 3$ and we are done.

Hence we may assume that  $X^\vee$ is a plane.

By  Lemma \ref{dp},  Propositions \ref{special} and \ref{insep},  either the number of singular 
points is at most 12, or  $ \deg(\ga) \deg (X^{\vee}) \geq 4$, or we are in the cases
where 
$$ X = \{ (x,z) | z^2 x_1^2 + B(x)=0\},$$ or 
 $$ X = \{ (x,z) | z^2 x_1x_2 + B(x)=0\}.$$
 
The former  case was dealt in Step I of Proposition 3 of Part I, showing that $X$ has at most 8 singular points,
and in this case Example 10 ibidem shows that $ \deg (\ga) \geq 4$.

In the latter case   Step II of Proposition 3 of Part I shows that $X$ has at most 13 singular points; and that it has exactly 
13 points only if it has 12 nodes (corresponding to the points of the plane 
where $ B_3= B_1 x_1 + B=0$), and a biplanar singular point (at $x=0$): 
hence also in this case the minimal resolution is a K3 surface.

\end{proof}

The following result improves upon part (V) of Proposition \ref{gaussestimate}.
\begin{cor}\label{14}
If  $\nu\geq 14$ all the singularities are either nodes or biplanar double points.
\end{cor}

\begin{proof}
Recall the basic inequality 
$$ 36 - \deg(\ga) \deg(X^{\vee} ) \geq    2 \nu + b + 6 u .$$

We are claiming $u=0$ if $\nu \geq 14$, hence it suffices to recall that we saw in the previous proposition
that $ \deg(\ga) \deg(X^{\vee}) \geq 3$.

\end{proof}

\section{Proof of the main theorem \ref{theo}  -- general bound }
\label{s:proof}

Throughout this section until \ref{ss:aux},
we assume that $X$ is a normal quartic surface with $\nu \geq 15$ singular points 
in order to establish a contradiction and prove the general bound of Theorem \ref{theo}.
We use the following result 
which will follow  from Propositions \ref{prop:>14} 
 and \ref{prop:cusps}
(to be proven in Section \ref{s:quasi} using the theory of elliptic and quasi-elliptic fibrations on K3 surfaces).

\bigskip

\begin{main-claim}
\label{main-claim}
If $X$ has $\nu \geq 15$ singular points, then, for each pair $P_i, P_j$ of singular points of $X$,
 the 
 line $L_{ij}^{\vee}$  dual to $L_{ij} : = \overline{P_i  P_j}$ is
contained in the dual surface $X^{\vee}$.
\end{main-claim}

\bigskip

\subsection{The main claim implies the general bound of Theorem \ref{theo}}
\label{ss:pf-thm}

\smallskip

It will suffice to show that:

\begin{claim}  
\label{claim:skew}
In the above setting,
$X^{\vee}$ contains two skew lines and  $7$ distinct coplanar lines.
\end{claim}

Indeed the claim implies that  $X^{\vee}$ is a surface of degree $\geq 7$, and by the Gauss map estimate
\ref{eq:deg} of Proposition
\ref{gaussestimate} we have
$$ 36 - 7 \deg(\ga)  \geq    2 \nu + b + 6 u,$$
hence $\nu \leq 14$ as announced.
 \qed
 
 \subsection{Proof of Claim \ref{claim:skew}}
 \label{ss:pfofclaim}
 
  We observe first that if a line  $L_{ij}$ passes through a third singular point $P$ of $X$,
 then it is contained in $X$, and the planes  $H \supset L_{ij}$ cut $X$ in the line  $L_{ij}$ plus a cubic $C$
 meeting  $L_{ij}$ in the three points $P, P_i, P_j$.
 
 Hence there cannot be $4$ collinear singular points: because then $C$ would contain  $L_{ij}$ and  $L_{ij} \subset$ Sing$(X)$,
 contradicting the normality of $X$.
 
 
 We show now that each plane contains at most 
$6$ singular points of $X$.

In fact, if the plane is the plane $z=0$, and the equation of $X$ is $$B(x) + z G(x) \mod ( z^2 ),$$
the singular points on the plane are the solutions of
$$ z = \nabla B(x)= B(x) = G(x)=0.$$
A reduced plane quartic has at most $6$ singular points. If the quartic is non-reduced, and 
 $B(x) = q(x)^2$, then the singular points 
are the solutions of $ z =q(x)  = G(x)=0$ and they are at most $6$ by the theorem of B\'ezout
and since $X$ is normal.

The case where $\{ x \mid  B(x)=0\}$ consists of  a double line and a reduced conic leads to at most one
singular point outside the line, hence at most $4$ singular points in the plane.

 Whence, if $\nu \geq 7$, there are $4$ linearly independent singular points of $X$, and we have found 
  two skew  lines  $L_{ij},  L_{hk}$: likewise the dual lines are skew.
  
  Assume now that $\nu \geq 15$ and consider all the lines of the form  $L_{1j}$: these are at least $7$,
   since at most 3 singular points are collinear,
  and the dual lines are contained in the plane dual to the point $P_1$.
  
  \qed

  \subsection{Propositions \ref{prop:>14} and \ref{prop:cusps} imply the Main Claim \ref{main-claim}}
  \label{ss:main-claim}

  Since we assume  $\nu\geq 15$, we can apply Proposition \ref{prop:>14} 
   to show that each pair $(P_i, P_j)$ induces a quasi-elliptic fibration.
 By the degree estimate in Proposition \ref{gaussestimate}, all singularities are nodes or biplanar double points,
 so Proposition \ref{prop:cusps} proves
 that 
    the pencil of planes containing $L_{ij}$ yields a line $L_{ij}^{\vee}$ contained in  $X^{\vee}$.

    \qed

\subsection{Auxiliary results}
\label{ss:aux}

We establish here, with  similar arguments, two easy results for later use.
To this end, we distinguish whether two given singular points $P_1, P_2$
 are collinear with a third singularity or not (in the latter case we call $P_1, P_2$ companions).
Recall that in the first case, the line $L=\overline{P_1P_2}$ is contained in $X$,
and each plane containing $L$ contains at most 6 singularities.

\begin{lemma}
\label{lem:companions}
If $\nu\geq 9$, then there is a singular point with two companions.
\end{lemma}

\begin{proof}
Assume to the contrary that each singularity has at most one companion.
Take a point $P_1$  and three collinear pairs, say $P_1 , P_2 , P_3 \in L$
$P_1 , P_4 , P_5 \in L_1$, $P_1 , P_7, P_8 \in L'$. 

Let $H$ be the plane  containing $P_1, P_2 , P_3, P_4 , P_5,$
and let $H'$ be the plane  containing  $P_1, P_2 , P_3, P_7 , P_8$. 
These are different, since each plane contains at most $6$
singular points.

By assumption, we may assume without loss of generality  that $P_4$ is not companion of $P_2$, 
hence there is $P_6$ collinear with $P_2, P_4$, 
so that $P_2,P_4,P_6 \in L_2\subset X$. 
At this stage we have obtained $6$ singular points (the maximum) in the plane $H$, and we observe that
$P_3$ is not companion of  $P_4$  or $P_5$. 

Hence we get 4 lines 
  $$X \cap H =  L+L_1+L_2+L_3 ,$$ 
  where $L_3$ must contain the   singular points $P_3, P_6$ and thus also $P_5$. Thereby we reach  the conclusion that $ P_1$ is companion of $P_6$.

We establish now a contradiction as follows.
Playing the same game for  the other  plane $H'$,
we find  another companion of $P_1$, call it  $P_9$.

Since $P_9 \in H' $, while $P_6\not\in H' $ (since $ H \cap  H' = L $) we have found two different companions for $P_1$, 
and we have reached a contradiction.

\end{proof}

\begin{prop}
\label{prop:deg8}
If $X$ has $\nu =14$ singular points and, for each pair $P_i, P_j$ of singular points of $X$,
 the pencil of planes containing the line $L_{ij} = \overline{P_i  P_j}$ yields a line $L_{ij}^{\vee}$
contained in the dual surface $X^{\vee}$, then the degree of the dual surface is at least $8$.
In particular,
the singular points are just $14$ nodes.
\end{prop} 
\begin{proof}
By the Gauss estimate it suffices to show the first assertion, and since $X^{\vee}$ contains at least two skew lines, it suffices to show that
it contains at least $8$ coplanar lines.
But this follows from Lemma \ref{lem:companions}
as there is a singular point $P_1$ on $X$ collinear with at most 5 pairs of singularities,
thus companion to at least 3, yielding a total number of at least 8 coplanar  lines on $X^\vee$.

\end{proof}

We will use Proposition \ref{prop:deg8} later, and we observe that 
 a weaker form suffices, where there is one point $P_1$ with the property of the proposition 
 holding for all lines $\overline{P_1 P_j}$, 
if   $X^\vee$ contains a  line 
skew to (one of) the 8 dual lines from $P_1$.

\section{Genus one fibrations}
\label{s:g=1}

We shall now invoke some results from the  theory of genus one fibrations on K3 surfaces
in order to achieve  the proof of Propositions   \ref{prop:>14} and \ref{prop:cusps}.

These will also be used for the proof of the other parts of Theorem \ref{theo}.

\bigskip

Let $X$ be a projective K3 surface.
{Let $L \in\Pic(X)$ be a divisor class with $L^2 \geq -2$; then, by Riemann-Roch, $\chi(L) \geq 1$
hence $L$ or $-L$ is effective. Hence let us assume that $ L $ is linearly equivalent to an effective divisor $D$.
If $D^2 =0$, then the linear system $|D|$ has dimension $\geq 1$, and we can write $ |D| = |M| + \Psi$,
where $\Psi$ is the fixed part. Clearly then $\Psi = \sum_i E_i$ where each $E_i$ is an irreducible curve with $E_i^2=-2$.

Since 
\begin{eqnarray}
\label{eq:0=}
0 = D^2 = M^2 + D \Psi + M \Psi , \ M^2 \geq 0, \ M \Psi \geq 0,
\end{eqnarray}
we have  $ D \Psi < 0$, or $\Psi=0$. Because, if $ D \Psi \geq  0$ and $\Psi>0$, then
\eqref{eq:0=} implies $M^2 =  D \Psi = M \Psi =0$,
hence $\Psi^2 =0 \Rightarrow \Psi=0$, the intersection form being negative definite by Zariski's lemma 
on the divisor $\Psi$: because  $\Psi$  is contained in the fibres of the fibration associated to $|M|$, 
there are no  multiple fibres, and $\Psi$  does not contain any full fibre (else, it would not be the fixed part).

The conclusion is that either $|D|$ has no fixed part or  there is $E_1$ such that $ D E_1 < 0$, hence reflection in the $(-2)$-curve $E_1$
produces a new divisor class $$D'' : =  D + (D E_1) E_1$$
 such that $(D'')^2=0$. The system $|D''|$ has dimension $\geq 1$, and  
 since the degree of $D''$ is smaller than the degree of $D$, the process terminates producing a base point free
 system  $|D'|$, with $(D')^2=0$, hence $|D'|$ is a pencil  of genus $1$ curves. We may also assume that $D'$ is primitive, so that $D'$
 is indeed a fibre of a fibration $ f : X \ra \PP^1$.

\medskip

If the general fibre is smooth, we call the fibration elliptic and we may further   distinguish 
whether the fibration admits a section or not.}
In characteristics $2$ and $3$, however, the general fibre may also be a cuspidal cubic curve
whence the fibration is called quasi-elliptic.

 Examples are given by sparse Weierstrass forms;
more precisely, in terms of the general equation \eqref{eq:WF} which shall be recalled later,
 those forms which do not contain  terms linear in $y$ (in characteristic $2$)
or all of whose terms have degree 0 or 3 in $x$ (in characteristic $3$).

In particular, quasi-elliptic surfaces over $\PP^1$ are unirational and thus supersingular ($\rho=b_2$)
 by \cite{shiodass}
which makes them quite special 
(see \cite{rudakov-shafarevich}, for instance).

\begin{rem}
\label{rem:-2}
 Any $(-2)$ curve $C$ on $X$ which is perpendicular to $D'$ features as a fibre component of $|D'|$
(but the analogous statement for $(-2)$-curves orthogonal to $D$
is surprisingly subtle in case there is some base locus involved.
 We will come back to this problem in part III.
\end{rem}

\begin{rem}
\label{rem:pencil}
In general, given an effective divisor $D$ with $D^2=0$, $|D|$ need not be a pencil, the easiest example being
the one where $D$ consists of a genus $2$ curve and a disjoint $(-2)$-curve, that is,
$$ D = M + E, M^2 = 2, E^2=-2, M E=0,$$
here  $M-E$ gives the desired pencil.

A sufficient condition for $\dim |D|=1$ is that the divisor $D$ is numerically connected, that is, 
any decomposition $ D =  A + B$, where $A,B$ are effective, satisfies $ AB \geq 1$.

Because, by the exact sequence
$$ 0 \ra \hol_X \ra \hol_X (D) \ra \hol_D(D)\ra 0$$
we have $H^1 (\hol_X (D))=0$ unless $h^1 (\hol_D (D))\geq 2$.
Since $h^1 (\hol_D (D))= h^0 (\hol_D),$
and  $h^0 (\hol_D)=1$ if $D$ is numerically connected, \cite {franchetta}, \cite{ram}, our claim follows.

In this case, $M^2=0$, and $D$ could, for instance, consist of a fibre plus a $(-2)$-curve which is a section,
$$ D = M + E, \;\; M^2 = 0, \;\; E^2=-2, \;\; M E=1.$$
\end{rem}

\subsection{ Disjoint smooth rational fibre components}
For later use, 
let us record some rather special features of elliptic fibrations in characteristic 2.

\begin{prop}
\label{lem:12}
 In characteristic 2, on an elliptic K3 surface the singular fibres contain at most 12 disjoint $(-2)$-curves.
\end{prop}

At first, this result may seem rather surprising,
since usually, i.e.\ outside characteristic $2$, 
elliptic fibrations allow for as many as 16 disjoint $(-2)$-curves.
This happens in the case of 4 fibres of Kodaira type I$^*_0$,
each containing 4 disjoint $(-2)$-curves
-- for instance,  on the Kummer surface  of a product of two elliptic curves.

\begin{proof}
 What prevents the same as above to happen in characteristic $2$ 
 is  the fact
 that all additive fibres, except for Kodaira types IV, IV$^*$, come with wild ramification by \cite{SSc}.
 
 More precisely, there still is a representation of the Euler-Poincar\'e characteristic of the elliptic K3 surface $X$ as a sum over the fibres:
 \[
 24 = e(X) = \sum_v (e(F_v) + \delta_v).
 \]
 Here $\delta_v$ denotes the index of wild ramification,
studied in more generality in \cite{Deligne-wild}.
On an elliptic surface, it can be computed as the difference of the Euler number 
$e(F_v)$ and the local multiplicity of the discriminant
which  is the equation for the  singular fibres and may be computed on the Jacobian by \cite[p.348]{CDL}.
The bounds for $\delta_v$ in the next table have been taken from \cite[Prop.~5.1]{SSc}. 
Note that the number of components $m_v$ is the index of the Dynkin type plus one, while, except in the first case, the Euler number is $m_v+1$.
 The table also collects the maximal number $N_v$ disjoint (-2)-fibre components, to be computed below.

 \begin{table}[ht!]
 \begin{tabular}{c||c|c|c|c|c|c|c|c|c}
 fibre type & I$_n$ & II & III & IV & I$^*_n (n\neq 1)$ & I$^*_1$ & IV$^*$ & III$^*$ & II$^*$\\
 \hline
 \hline
 Dynkin type & $A_{n-1}$ & $A_0$ & $A_1$ & $A_2$ & $D_{n+4}$ & $D_5$ & $E_6$ & $E_7$ & $E_8$\\
 \hline
 $m_v$ & $n$ & 1 & 2 & 3 & $n+5$ & 6 & 7 & 8 & 9\\
 \hline
 $\delta_v$ & 0 & $\geq 2$ & $\geq 1$ & 0 & $\geq 2$ & 1 & 0 & $\geq 1$ & $\geq 1$\\
  \hline
 $e(F_v)$ & $n$ & 2 & 3 & 4 & $n+6$ & 7 & 8 & 9 & 10\\
  \hline
 $N_v$ & $\lfloor \frac{n}{2}\rfloor$ & 0 & 1 & 1 & $4 + \lfloor \frac{n}{2}\rfloor$ & 4 & 4 & 5 & 5\\

 \end{tabular}
 \end{table}
 
 For the convenience of the reader, we also include the dual graphs of the fibres
 in terms of the extended Dynkin diagrams $\tilde A_n, \tilde D_k\, (k\geq 4)$.
 For fibre types IV$^*$, III$^*$, II$^*$, we only give the Dynkin diagram $E_l\, (l=6,7,8)$
 for sake of a unified presentation. For these types 
the fibre is obtained by adding another   fibre component $e_0$ 
  adjacent  to the vertex $e_1$ in case $E_6$, resp.\  $e_2$ in case $E_7$, resp.\  $e_8$ in case $E_8$.
 
 In total, the  simple fibre components  (i.e.\ those having multiplicity 1 in the fibre) are:
 \begin{enumerate}
 \item[$\tilde A_n$]
 all components,
 \item[$\tilde D_k$]
 the exterior components,
 \item[$\tilde E_l$]
  $e_0, e_2, e_6$ ($l=6$) resp.\ $e_0, e_7$ ($l=7$), resp.\ $e_0$ ($l=8$).
 \end{enumerate}

\begin{figure}[ht!]
\setlength{\unitlength}{.6mm}
\begin{picture}(80,20)(5,25)
\multiput(3,32)(20,0){5}{\circle*{1.5}}
\put(3,32){\line(1,0){43}}
\put(83,32){\line(-1,0){23}}

\put(2,25){$a_1$}
\put(22,25){$a_2$}
\put(49,32){$\hdots$}
\put(82,25){$a_{n}$}
\put(3,32){\line(4,1){40}}
\put(83,32){\line(-4,1){40}}
\put(43,42){\circle*{1.5}}
\put(45,45){$a_0$}
\put(-50,35){$(\tilde A_n)$}

%
%

\end{picture}
\end{figure}

\begin{figure}[ht!]
\setlength{\unitlength}{.6mm}
\begin{picture}(100,25)(5,-4.5)
%
\put(-40,7){$(\tilde D_k)$}
\multiput(23,8)(20,0){4}{\circle*{1.5}}
\put(23,8){\line(1,0){23}}
\put(83,8){\line(-1,0){23}}
\put(49,8){$\hdots$}
\put(83,8){\line(2,1){20}}
\put(83,8){\line(2,-1){20}}
\put(103,18){\circle*{1.5}}
\put(103, -2){\circle*{1.5}}

\put(23,8){\line(-2,1){20}}
\put(23,8){\line(-2,-1){20}}
\put(3,18){\circle*{1.5}}
\put(3, -2){\circle*{1.5}}

\put(-6,18){$d_0$}
\put(-6,-2){$d_1$}
\put(22,1){$d_2$}
\put(78,1){$d_{k-2}$}
\put(106,17){$d_{k-1}$}
\put(106, -2){$d_{k}$}


\end{picture}
\end{figure}

\begin{figure}[ht!]
\setlength{\unitlength}{.6mm}
\begin{picture}(120,30)(3,-40)
%
\multiput(3,-32)(20,0){7}{\circle*{1.5}}
\put(3,-32){\line(1,0){83}}
\put(123,-32){\line(-1,0){23}}
\put(2,-39){$e_2$}
\put(22,-39){$e_3$}
\put(42,-39){$e_4$}
\put(62,-39){$e_5$}
\put(43,-32){\line(0,1){20}}
\put(43,-12){\circle*{1.5}}
\put(46,-13){$e_1$}
\put(89,-32){$\hdots$}
\put(122,-39){$e_l$}

\put(-31,-27){$(E_l)$}

\end{picture}
\end{figure}

Case by case, this allows us to compare the maximal number $N_v$ of disjoint (-2)-fibre components
with the contribution  to the Euler-Poincar\'e characteristic, see the above table.

Overall, we find
\begin{eqnarray}
\label{eq:N_v}
N_v \leq \frac 12 \lfloor e(F_v) + \delta_{v}\rfloor
\end{eqnarray}
and thus
\begin{eqnarray}
\label{eq:eqeq}
\sum_v N_v 
\leq   \sum_v \frac 12 \lfloor e(F_v) + \delta_{v}\rfloor
\leq \frac 12 \sum_v (e(F_v) + \delta_{v}) = 12.
\end{eqnarray}
This yields the desired  inequality and proves our assertion.

\end{proof}

\begin{rem}
\label{rem:ineq}
If equality holds at each step of the chain of inequalities 
$$N_v \leq \frac 12 \lfloor e(F_v) + \delta_{v}\rfloor \leq \frac 12 (e(F_v) + \delta_{v}),$$
then $\de_v$ attains its minimum value, and the multiplicity $(e(F_v) + \delta_{v}) $ is an even number,
hence we get only the types \[
\mathrm I_{2n} \; (n>0), \;\; \mathrm I_{2n}^*\; (n\geq 0), \;\;  \mathrm I_1^*,\;\;   \mathrm{IV}^*,\;\;\mathrm{III}^*.
\]

\end{rem}

%

\begin{cor}
\label{cor:12fibres}
If the fibres of an elliptic K3 surface in characteristic $2$ contain 12 disjoint $(-2)$-curves,
then the only possible singular fibre types are
 (with minimum possible $\delta_v$ each)
 \[
\mathrm I_{2n} \; (n>0), \;\; \mathrm I_{2n}^*\; (n\geq 0), \;\;  \mathrm I_1^*,\;\;   \mathrm{IV}^*,\;\;\mathrm{III}^*.
\]
\end{cor}

\begin{proof}
This is a direct consequence of the proof of Proposition \ref{lem:12}
since all the inequalities in \ref{eq:eqeq} actually have to be equalities
(in particular the  same must hold for \ref{eq:N_v} at each $v$,
as in Remark \ref{rem:ineq}).

\end{proof}

A close inspection of the fibres in the proof of Proposition \ref{lem:12}
allows us even to rule out higher ADE-types:

\begin{cor}
\label{cor:ADE}
If the  fibres of an elliptic K3 surface in characteristic 2 support 12 disjoint ADE-configurations
of $(-2)$-curves, then each has type $A_1$.
\end{cor}
\begin{proof}
These 12 disjoint ADE-configurations produce at least 12 disjoint $(-2)$-curves, hence we may apply the
previous corollary and check directly.
\end{proof}

\subsection{ Connection with supersingularity}

 To relate  with Theorem \ref{theo},
especially with the statement about supersingular K3 surfaces,
we provide the next result which concerns the case of exact equality in Proposition \ref{lem:12}.

\begin{prop}
\label{lem:=12}
Let $X$ be an elliptic K3 surface such that there are 
 12 disjoint $(-2)$-curves contained in the fibres.
Then $X$ is supersingular or 
there are two additive fibres.
\end{prop}

 Note that the  fibres in Proposition \ref{lem:=12} are the fibres of  Corollary \ref{cor:12fibres}:
either $\mathrm I_{2n}$, or additive fibres which  
are  non-reduced.  This will be of great use in what follows.

\begin{rem}
\label{rem:occur}
(i)
It is easy to see that both cases of Proposition \ref{lem:=12} can occur:  the first one via 
inseparable base change from 
rational elliptic surfaces {(see \cite[p.\ 342]{MWL}, Proposition 12.32)} 
the other one (as in characteristic zero!)
 by taking the Kummer surface of the product of two elliptic curves  (both not supersingular):
 here there are  two singular fibres of Kodaira type $\mathrm{I}_4^*$ by \cite{shioda}.

(ii)
The second case of Proposition \ref{lem:=12} encompasses the case where there are 12 disjoint
$(-2)$-curves contained in the fibres and the   j-invariant is constant, since then every reducible fibre is additive, and if
there were a single reducible fibre, it would have type $\mathrm I_{16}^*$,  which
is impossible by \cite{S-max}.
\end{rem}

\begin{proof}[Proof of Proposition \ref{lem:=12}]
If the singular fibres contain 12 disjoint $(-2)$-curves,
then by the proof of Proposition \ref{lem:12},
both inequalities in \ref{eq:eqeq}
are in fact equalities,
with fibre types given in Corollary \ref{cor:12fibres}.

Hence $\delta_v$ attains the minimal possible value $ \delta_{v}(min)$  and $e(F_v) + \delta_{v} = 
e(F_v) + \delta_{v}(min) $ is always even.

Since $e(F_v) + \delta_{v}$  is exactly the vanishing order of the discriminant $\Delta$ at $v$ by  \cite{Ogg},
we find that $\Delta$ is a square in $k(t)$.

\subsubsection{The Jacobian fibration}
We now switch to the Jacobian $J$ of $X$ -- another elliptic K3 surface,  since it shares the same
invariants  of $X$ by \cite[Cor.\ 5.3.5]{CD}.
Note that $J$ also has the same Picard number as $X$,
but, by definition, $J$  has a section while $X$ may not.

By \cite[Theorem  5.3.1]{CD} $J$ and $X$ share the same  singular fibres  
(and by \cite[p.348]{CDL} also the same $\Delta$ and $\delta_v$ (minimal!))
 since, by virtue of the canonical bundle formula (Theorem 2 of \cite{bm}), there are no multiple fibres. 
 
In terms of a minimal Weierstrass equation for $J$,
\begin{equation}
\label{eq:WF}
y^2 + a_1 xy + a_3 y = x^3 + a_2 x^2 + a_4 x + a_6, \;\;\; a_i\in k[t], \deg(a_i)\leq 2i,
\end{equation}
there are essentially two options for $a_1$ (since $a_1\equiv 0$ forces all singular fibres to be additive
and is thus covered by the second alternative of Proposition \ref{lem:=12}, 
cf. Remark \ref{rem:occur} (ii)), up to M\"obius transformations:
\[
a_1 = t \;\;\; \text{ or } \;\;\; a_1 = t^2.
\]
In the first case, we can argue directly with the general expression of the discriminant,
\begin{eqnarray}
\label{eq:Delta}
\Delta = a_3^4+a_1^3a_3^3
+ a_1^4a_4^2
+a_1^4a_2a_3^2
+ a_1^5a_3a_4
+a_1^6a_6.
\end{eqnarray}
Notably, if $a_1=t$, then  this reads modulo $t^4$
\[
\Delta \equiv a_3(0)^4 + a_3(0)^3t^3 \mod t^4,
\]
so $\Delta$ can only be a square if $a_3(0)=0$ which makes the fibre at $t=0$ singular and in fact additive.
By symmetry, the same reasoning applies at $t=\infty$, so there are  two additive fibres and we 
reach the second alternative of this proposition.

\subsubsection{ Normal forms for additive fibre types}

There  remains to study the case $a_1=t^2$.
 We start arguing with the minimality of $\delta_0$ to reduce to just 3 cases.
 
If there is a singular fibre at $t=0$  (then $a_3$ vanishes at $t=0$ and we have an additive fibre), 
then we can use Tate's algorithm to develop a normal form for the fibre \cite{Tate}, \cite[IV.9]{Si3}.

For fibres of type I$^*_{2n}$, the normal form is
\begin{eqnarray}
\label{eq:2n^*}
y^2 + t^2 xy + t^{n+2}a_3' y  & = & x^3 + ta_2' x^2 + t^{n+2}a_{4}'x + t^{2n+4} a_6' 
\end{eqnarray}
with $t\nmid a_2'a_4'$;
here we have used  Steps 6 and Step 7 in \cite[IV.9]{Si3},  pages 367-368. 
For $n=0$ we use indeed Step 6, and  the fact that the auxiliary polynomial $P(T)$ in loc.\ cit.\
has three distinct roots to infer that $t\nmid a_2'a_4'$ after locating one root at $T=0$. 
For $ n =1$  the assertions are proven
in Step 7, page 367; for higher $n$ one proceeds by induction on $n$, 
 see  line 8 of page 368 concerning the assertion 
  on the divisibility of $a_3, a_4, a_6$
 going up in each induction step.
 Note that by the argument in loc.\ cit., the divisibility of $a_6$ grows in fact by two in each of our steps.
This shows that $t^{n+2}\mid a_3, a_4$ and $t^{2n+3}\mid a_6$ and then 
 a translation in $x$  ensures that 
 indeed  $t^{2n+4}\mid a_6$ as claimed.

Substituting into \ref{eq:Delta} gives 
\[
\Delta = t^{4n+8}a_3'^4 + t^{3n+12}a_3'^3 + t^{2n+12} a_4'^2 + h.o.t.,
\]
whence, for the wild ramification 
$$\de_0 = \ord (\De) - e (F) = \ord(\Delta) - (2n+6) \geq 2n+2,
$$
we have $\delta_0\geq 4$ for $n>0$.
Since Corollary \ref{cor:12fibres} requires minimal wild ramification $\delta_0=2$, 
this leaves only fibres of type I$_0^*$ among all fibre types I$_{2m}^*$.

For a fibre of type I$^*_{1}$, the normal form 
arises from an additional vanishing condition at $a_4$ compared to \eqref{eq:2n^*},
again by \cite[IV.9, Step 7]{Si3}:
\begin{eqnarray*}
\label{eq:2n+1^*}
y^2 + t^2 xy + t^{2}a_3' y = x^3 + ta_2' x^2 + t^{3}a_{4}'x + t^{4} a_6' \;\;\; \text{ with } \;\; t\nmid a_2'a_3'.
\end{eqnarray*}

Then fibre type IV$^*$ is given by further imposing $t^2\mid a_2$ by \cite[IV.9, Step 8]{Si3}, still with $t\nmid a_3'$
(in agreement with $\delta_v=0$).
Meanwhile a fibre of type III$^*$ imposes additional vanishing conditions 
$t^3\mid a_3,\; t^5\mid a_6$, but $t^4\nmid a_4$ \cite[IV.9, Step 9]{Si3}.
Substituting into \ref{eq:Delta} gives 
\[
\Delta = t^{12}a_3'^4 + t^{14} a_4'^2 + t^{15}a_3'^3 + h.o.t.,
\]
so in particular $\delta_0\geq 3$, ruling out fibre type III$^*$ by Corollary \ref{cor:12fibres} again.

To sum it up,
the only  additive  fibre  types remaining 
from Corollary \ref{cor:12fibres}
 are
I$_0^*$, I$_1^*$ and IV$^*$.
In each  case, one can easily parametrize all K3 surfaces with such a given fibre and square discriminant,
 starting from the above normal form.
It should be noted that for these types the normal form can be derived by means of a linear transformation 
\begin{eqnarray}
\label{eq:adm}
(x,y) \mapsto (x+\alpha_4, y+\alpha_2 x + \alpha_6)
\end{eqnarray}
with $\alpha_i\in k[t]$ of degree at most $i$;
in particular, the degree bounds of \eqref{eq:WF} are preserved.

\subsubsection{ Conditions for the  discriminant to be a square}
For type I$_0^*$, \ref{eq:2n^*} leads to the discriminant
%
\[
\Delta = t^8(a_3'^4+t^4a_3'^3
+ t^4a_4'^2
+t^5a_2'a_3'^2
+ t^6a_3'a_4'
+t^8a_6')
\]
where, by the minimality of wild ramification, $t\nmid a_3'$.
 Modulo square summands, this simplifies as
\[
\Delta \equiv t^{12}(a_3'^3
+ta_2'a_3'^2
+ t^2a_3'a_4'
+t^4a_6') \mod   k[t]^2.
\]
Write $a_i' = \sum_j a_{i,j}'t^j$.
Then the condition that $\Delta$ is a square, i.e.\ that all odd degree coefficients vanish, determines

\begin{itemize}
\item
the odd degree coefficients of $a_6'$ in terms of the coefficients of
the other forms $a'_m$ (looking at the coefficients of $\Delta$ at $t^{17},\hdots,t^{23}$).
\item
$a_{2,0}' = a_{3,1}'$ (from the $t^{13}$-coefficient);
\item
$a_{4,1}' = (a_{2,2}'a_{3,0}'^2 + a_{3,0}'^2a_{3,3}' + a_{3,1}'a_{4,0})/a_{3,0}'$ (from the $t^{15}$-coefficient).
\end{itemize}

In particular, we find that the family
of elliptic K3 surfaces with a fibre of type I$_0^*$ with wild ramification of index 2
and all other singular fibres of type I$_{2n}$ (generically 8 I$_2$'s)   is irreducible.

Its moduli dimension, equal  to $7$,  is obtained by comparing
the degrees
\[
\deg(a_3')\leq 4, \;\; \deg(a_2')\leq 3,\;\; \deg(a_4')\leq 6, \;\; \deg(a_6')\leq 8
\]
(these bounds follow from the degree bounds in \eqref{eq:WF} and from  \eqref{eq:2n^*}),
 against M\"obius transformations $t\mapsto ut/(\varepsilon t +1) \; (u \in k^\times, \varepsilon\in k)$
and the following variable transformations preserving the shape of \ref{eq:2n^*}
(since $\Delta$ being a square is automatically preserved):
\begin{eqnarray}
\label{eq:adm'}
(x,y) \mapsto (u^4x+t^2\beta_2, u^6y+t\beta_1 x + t^2\beta_4)
\end{eqnarray}
where the degree of each polynomial $\beta_i\in k[t]$ is at most $i$.

\subsubsection{ Conclusion of proof using number of moduli}

Any smooth K3 surface arising from a  member of the above family satisfies
\[
\rho \geq 2 + 8 + 4 = 14
\]
by the Shioda--Tate formula
where the first entry comes from the zero section and the fibre,
 the second from the semi-stable fibres  (each of type I$_{2n}$ for some $n\in\NN$,
 hence contributing $2n$ to the Euler--Poincar\'e characteristic and $2n-1$ to the Shioda--Tate formula)
and the third from the fibre at $t=0$ (contributing $8$ to the Euler--Poincar\'e characteristic,
including wild ramification, and $4$ to the Shioda--Tate formula).
If a very general member 
 were not supersingular,
then it would deform in a $6$-dimensional family 
as in \cite[Prop.\ 4.1]{LM} (based on \cite{Deligne})
but this is exceeded by our moduli count.
Hence the whole family is supersingular as claimed.

We pass now to the case of a fibre of type I$_1^*$ or of  type IV$^*$.
As explained before, the K3 surfaces with a fibre of type I$_1^*$ are contained in the subfamily
where  $t^3\mid a_4$ (while for I$_0^*$ we simply had $t^2\mid a_4$) and  
 type IV$^*$ additionally requires $t^2\mid a_2$.
Each family  allows  the same transformations, so 
the moduli dimension is 6,  resp.\ 5.
But  $\rho$ generically goes up by 1 each time (promoting the root lattice 
at the special fibre from $D_4$ through $D_5$ to $E_6$),
so 
%
%
 the whole family is supersingular 
  again by  \cite[Prop.\ 4.1]{LM}.

If there is no additive fibre, then  the condition that $\Delta$ is a square gives 9 moduli for $\Delta$:  moreover
 the condition 
that $ a_1 = t^2$ reduces the number of moduli to 8, and one can show by  the same kind of arguments as above
that we have   an irreducible 8-dimensional family
of semi-stable elliptic K3 surfaces with 12 disjoint $A_1$'s  embedding into the singular fibres;
since $\rho\geq 2+12=14$,
  again by the formula of \cite[Prop.\ 4.1]{LM} the family is supersingular. 

\end{proof}

\begin{rem}
 Another possible argument of proof  is as follows: in each case we have an irreducible family of a certain dimension $k$,
and inside it we can construct a family   of  the same dimension $k$ of surfaces arising   via an  inseparable base change
from a rational elliptic surface. The surfaces are   thus  unirational, hence supersingular, and this shows directly that the
original family is a family of supersingular surfaces.

Indeed, starting from rational elliptic surfaces with singular fibre at $t=0$ of type
$\mathrm I_0$ (smooth supersingular), $\mathrm{II}, \mathrm{III}, \mathrm{IV}$, respectively, inseparable base change exactly results in a family of supersingular K3 surfaces
of the expected type and dimension.
Note that, since the elliptic fibrations  admit a 2-torsion section by \cite[p.342]{MWL},
the Artin invariants   \cite{artinSS} satisfy $\s \leq 9$.
\end{rem}

\section{Proof of  the main claim: there cannot be at least  15 singularities.}
\label{s:quasi}

In order to bound the number of singularities on a normal quartic $X\subset\PP^3$,
we shall use the theory of genus 1 fibrations  laid out in the previous section.

By Proposition  \ref{cor:9}, if $X$ has at least $13$ singular points ($\nu\geq 13$),
then  the singularities are rational double points
and the minimal resolution $S$ is a  K3 surface.

%
%
%
%
%
%
%
$S$  is endowed with the following divisors:
the pull-back $H$ of a plane section
and, for each pair of singular points, say $P_1, P_2$,     the respective fundamental cycles $D_1, D_2$ 
(see \cite{artin}), consisting of the 
exceptional curves with suitable multiplicities, and equal to the pull back of the maximal ideal at the singular point.

Then, since $D_i^2 = -2$,  
\[
E:= H-D_1-D_2
\] 
gives an effective isotropic class in $\Pic(S)$.

We have that the linear system $|E|$ is base point free if and only if 
the line $L=\overline{P_1P_2}$ is not contained in $X$: this is clear for the points of $S$ not lying over $P_1, P_2$;
moreover, since for each exceptional curve $C$ the intersection number $D_i C \leq 0$, $E$ has no fixed part
(it  was observed at the beginning of the previous section that the fixed part $\Psi$ satisfies, if non empty,  $ E \Psi < 0$)
hence it has no base points since $E^2=0$.

If instead the line $L$ is contained in $X$, denote still by $L$ the strict transform of the line and 
 replace $E$ by $E-L$, observing that $E L = -1$, hence $(E-L)^2 =0$,
and continue until we get a base point free pencil $|E'|$, which gives a morphism $ S \ra \PP^1$ whose fibres 
correspond to the planes
through $P_1, P_2$.

\begin{prop}
\label{prop:>14}
Let $X\subset\PP^3$ be a normal
quartic with at least  15 singularities.
Then every genus one pencil $|E'|$ arising from two singularities on $X$ as above
is quasi-elliptic. 
\end{prop}

\begin{proof}
 
Let $\nu\geq 15$ denote the number of singularities,  $P_1,\hdots,P_{\nu}$,
and let $C_i^{j}, j = 1, \dots, n(i)$ be the  irreducible exceptional curves lying above the point $P_i$.

Let us first assume that no $P_i \, (i>2)$  lies on $L$,
so that each lies on a unique plane through $P_1, P_2$;
hence the $C_i^j$'s  are components of the corresponding fibre of $|E'|$
(as in Remark \ref{rem:-2}).
Then  the fibration $|E'|$ has $\nu-2>12$ disjoint smooth rational fibre components (the $C_i^j$).

If, on the other hand, there is a third singularity on $L$, say $P_3$,
then this implies not only that $L\subset X$, but also that $L$ appears as a multiple component
of $X\cap H$ for a unique plane $H$ (just take a  point $P \in L$ which is a smooth point of $X$,
and let $H$ be the tangent plane to $X$ at $P$: then $ H \cap X \geq 2L$).

This implies that $L$ is a component of the  fibre corresponding to $H$,
and, together with $C_4^1,\hdots,C_{\nu}^1$, we obtain $\nu-2>12$ disjoint smooth rational fibre components
as before.

In both cases 
 the proposition follows then from Proposition \ref{lem:12}. 
 \end{proof}

 \begin{prop}
 \label{prop:cusps}
     Let $X\subset\PP^3$ be a normal quartic.
     Let $L$ be a line through two singular points of $X$ such that $X\cap L$ consists
     of nodes and  biplanar double points (and smooth points if $L\subset X$).
  If the fibration induced by $L$ is 
quasi-elliptic, then  the  line dual to $L$ is contained in the dual surface $X^{\vee}$. 
\end{prop}

 \begin{proof}
 We  consider  the curve $\Sigma_0 \subset S$ ($S$ is  the minimal resolution  of $X$ as usual)
 consisting of the horizontal divisorial part of the set 
 of singular points of the fibres,  the so-called curve of cusps.
 
 The first case  is when this curve is not  exceptional for the map 
 $$ \Phi: S \ra X \subset \PP^3;
 $$
 then we get a curve  on $X$ consisting  of singular points of the intersections $H \cap X$,
 where $H$ is a plane of the pencil through $L=\overline{PP'}$.
  Therefore the dual line $L^\vee$ is contained in $X^\vee$.
 
 \medskip

 The second case is where  $\Sigma_0$ is exceptional:
 we use for this Proposition 1, page 199 of \cite{bminv}, and denote as in loc.\ cit.\ $f : S \ra B$
 the quasi-elliptic fibration. At a  general point $Q' \in \Sigma_0$, the fibre $ F : = F_{f(Q')}$ has a cusp 
 and, if $t$ is a local parameter for $B$ at $f(Q')$, 
  the map is given by $ t = u (x^2 + y^3)$ 
 where $u$ is a unit in the formal power series ring which is the completion of the local ring $\hol_{S,Q'}$.
 
 Bombieri and Mumford show that there is a local parameter $\s$ such that  $\Sigma_0 = \{ \s=0\}$,
 and that $ (\Sigma_0  \cdot F )_{Q'} = 2$, so that $x,\s$ are local parameters for $S$ at $Q'$,
 and we can write $ y = \s + \la x $ plus higher order terms.
 
  Since we assume that the curve $ \Sigma_0 $ is contracted by the map $\Phi$, it follows that
  $\Phi$ has a local Taylor development
 which  contains only terms in the ideal generated by $y$.
 
 Hence  we are left only with monomials $y, y^2, xy, \dots$ whose 
 respective orders on the normalization of the fibre $F$ are: $2,4,5$.
 
 We conclude that the image of $F$  under $\Phi$ has a higher order cusp at a singular point $P''$
 lying in $L$. 
 
 By assumption, we can write the  equation of $X$ at $P''$ in local affine coordinates as
 $$ h: = xy +  \la z^2 + g (x,y,z) = 0 \;\;\;  (\lambda\in K),
 $$
 where $g$ has order $\geq 3$. 
 
 Since we want that the planes of the pencil  cut a cusp at $P''$, the quadratic part of the restriction of the equation $h$
  to the planes must be the square of a linear form,
hence in the projectivized tangent space we get lines  intersecting the exceptional conic $C$ with multiplicity two,
hence  lines  tangent to the conic;  from the equation $ xy +  \la z^2$ of the quadratic part follows that 
this   pencil is generated by the linear forms $x,y$.

  We claim now that, as in the first case,
    the pencil of planes through $L$ yields a line in the dual surface $X^{\vee}$.

 Because the Gauss map is given by $(y,x, 0,0) + h.o.t $,  and the  image of the exceptional divisor   in the dual surface is the pencil of planes $\mu_0 x + \mu_1 y =0$,
 which is exactly the pencil of planes containing $L$ by our previous argument.

\end{proof}

\begin{rem}
Both cases from the proof of the proposition actually occur (for the second case, it suffices that $g(x,y,z)$ above has order $ 4$).
\end{rem}

Note that 
 Propositions \ref{prop:>14} and \ref{prop:cusps} provide 
the missing ingredients for the proof of the Main Claim \ref{main-claim} in \ref{ss:main-claim}.
Thereby the proof of the first statement of Theorem \ref{theo} is now complete.

\section{14 singularities are nodes}

The aim of  this section is to prove the following result
which covers the second part of Theorem \ref{theo}:

\begin{theo}
\label{thm:14nodes}
Let $X\subset\PP^3$ be a normal quartic
with 14 singular points.
Then all singularities are nodes.
\end{theo}

\begin{proof}

The minimal resolution $S$ of $X$ is a K3 surface by Proposition \ref{cor:9},
and the singular points are  nodes or biplanar double points ($u=0$)

by Corollary \ref{14}.



Assume that we have a  singular point $P$ which is not of  type $A_1$.
Just like in the proof of Proposition \ref{prop:>14},
any genus one fibration $S\to \PP^1$ 
induced by two singular points admits 12 disjoint smooth rational curves in the fibres.
By Proposition \ref{lem:12} this is the maximum possible for an elliptic fibration.

If the fibration is not induced by $P$ and another singular point, 
 $P$ lies in exactly one fibre of $X$
and   the fundamental cycle is supported on the corresponding fibre of $S$.
Hence Corollary \ref{cor:ADE} implies that the fibration is quasi-elliptic.

Any singular point $Q\neq P$ 
thus admits at least 6 quasi-elliptic fibrations induced by a pair of singular points  $Q, Q'$
 which are nodes or biplanar double points.
Hence we infer  from Proposition \ref{prop:cusps} 
and  the proof of Proposition \ref{prop:deg8}
that $\deg(X^\vee)\geq 6$ and $b\leq 2$.
More precisely, by Remark \ref{rem:non-RDP},
$P$ can only have type $A_2$ or $A_3$,
and in the former case there may be a second singular point $P'$ of type $A_2$.

In fact, we can say more about the configuration of singularities relative to $Q$.
Namely $Q$ is collinear with at least 5 pairs of singularities (possibly including $P$),
for else it would induces at least 8 quasi-elliptic fibrations,
and $\deg(X^\vee)\geq 8$ would give a contradiction using \eqref{eq:deg}.

We pick one such pair not involving $P$,
say $Q, Q', Q''\in L\subset X$,
and consider  the induced quasi-elliptic fibration 
$$\pi: S \to \PP^1.
$$

The fibres are the cubics $C$ residual to $L$ in the respective plane $H$ containing $L$.
Except possibly for the cubic containing $L$ as a component, these cubics are all reduced,
since they meet $L$ in the three points $Q, Q', Q''$.
Recall moreover that the exceptional (-2)-curve resolving a node not on $L$ also appears always with multiplicity 1,
hence the only fibres of $\pi$ which may not be reduced are those containing exceptional curves lying 
above the singular points of type $A_n$ with $ n \geq 2$
and the one containing $L$.
Since $b\leq 2$, this makes for at most 3 fibres.

Since there are  5 pairs of singular points collinear with $Q$,
there has to be a pair of nodes left which lie on a reduced fibre
(since no plane can contain more than 6 singular points (cf.\ \ref{ss:pfofclaim}),
so all pairs of points $\neq (Q',Q'')$ collinear with $Q$ lie on different fibres).
 In particular, this reduced fibre has at least 4 components.
However, by \cite[Prop.\ 5.5.10]{CD}
the possible fibre types
of a quasi-elliptic fibration are a priori 
\begin{eqnarray}
\label{eq:fibres}
\mathrm{II}, \;\;
\mathrm{III}, \;\; \mathrm I_{2n}^*\; (n\geq 0), \;\;    \mathrm{III}^*,\;\;\mathrm{II}^*.
\end{eqnarray}
 Of these, only fibres of type $\mathrm{II}, \mathrm{III}$ are reduced, with one or two components.
 This gives the required contradiction.
 Hence all singularities of $X$ are nodes.

\end{proof}

\section{Proof of Theorem \ref{theo}: the non-supersingular case}
\label{s:non-ss}

To complete the proof of Theorem \ref{theo},
it remains to analyse the non-supersingular case.

\begin{prop}
\label{prop:nu<14}
Let $X\subset \PP^3$ be a normal quartic
such that a minimal resolution is not a supersingular K3 surface.
Then $X$ contains at most 13 singular points.
\end{prop}

\begin{proof}
By the general part of Theorem \ref{theo},
we only have to rule out: $X$ contains 14 singularities.
Assuming  this, all singular points are nodes by Theorem \ref{thm:14nodes}, 
and the minimal resolution $S$ is a K3 surface
(non-supersingular by assumption).
We continue to study the fibrations $\pi_{i,j}$ induced by pairs of nodes $(P_i, P_j)$.
The proof of Theorem \ref{thm:14nodes} shows that the fibres contain
12 disjoint $(-2)$-curves, so by Proposition \ref{lem:=12} 
there are two additive fibres; by Corollary \ref{cor:12fibres}, the possible types are
\begin{eqnarray}
\label{eq:types}
\mathrm I_{2n}^* (n\geq 0), \;\;  \mathrm I_1^*,\;\;   \mathrm{IV}^*,\;\; \text{ and } \;\;\mathrm{III}^*.
\end{eqnarray}
We distinguish three cases:

\subsubsection{}

\label{ss:3_collinear}
If there are 3 collinear nodes,
then they give sections of the induced fibration,
and the 11 exceptional curves above the other nodes embed into the negative definite root lattices
which are the orthogonal complements of the sections.
On the multiplicative fibres, this imposes no general restrictions,
but additive fibres can, by inspection of the singular fibres, as described in the proof of Proposition \ref{lem:12}, only support one disjoint smooth rational curve less.  This is because the sections necessarily intersect only the simple fibre components. Therefore  
the number of  disjoint rational curves not meeting one  of the three sections is at most
$N_v - 1$, where we recall that
\begin{eqnarray}
\label{eq:N_v'}
N_v \leq \frac 12 \lfloor e(F_v) + \delta_{v}\rfloor.
\end{eqnarray}
Hence, with two additive fibres, there can only be 10 disjoint $(-2)$-curves
supported on the orthogonal complement of the sections, contradiction.

\subsubsection{}
\label{ss:2_collinear}
Thus we may assume that there are no three collinear nodes
(i.e.,  any two nodes are companions).
Note that this implies that any 3 nodes lie on a unique plane.
If some connecting line is contained in $X$,
then the line and the two nodes give sections of the fibration,
with 12 disjoint $(-2)$-curves in the fibres.
Hence the argument from \ref{ss:3_collinear} applies
to establish a contradiction.

\subsubsection{}
\label{ss:double_conics}
We can therefore assume that no   line $\overline{P_iP_j}$  is contained in $X$.
We continue by restricting  the possible additive fibre types.
They arise from the quartic curve $X\cap H$ by blowing up the nodes in the plane $H$:
two of them give bisections of the fibration while the others give $(-2)$-fibre components,
 which cannot be such that their multiplicity in the fibre is $\geq 3$.

\smallskip

Only the additive fibre type I$^*_0$ can be realized of the possible types listed in \eqref{eq:types}
(as a double conic with 6 nodes;  see part III). The reason is based on the fact that  
 this is the only one with only one  component with multiplicity at least 2,
while the others have several components appearing with multiplicity at least 2, indeed at least 3 
components except for the case of I$^*_1$.

Indeed, the blow ups of nodes appear with multiplicity 1, hence the plane section $X \cap H$ must be non reduced. 
In particular there are at most 2 components appearing with multiplicity at least 2. 

To exclude the case of I$^*_1$, we need to exclude that $X \cap H$ consists of two double lines.
 In this case there are at most 4 nodes in $H$, since the intersection point of the two double lines cannot be a node, and there are no 3 collinear nodes by assumption,
hence the number of irreducible components of the fibre is at most 4, a contradiction.

\smallskip

 Therefore  each fibration $\pi_{i,j}$ admits two such fibres.
This turns out to be too restrictive:
 in fact, we have seen that for each pair $\sP$ of nodes, there are exactly two planes $\pi$
containing the pair, each containing six nodes. 

Consider then the number of pairs as above $(\sP, \pi), \sP  \subset \pi$. The number is therefore $(13)\cdot 14$.
But since each such plane $\pi$ contains exactly $15$ such pairs $\sP$, we have obtained a contradiction.

\end{proof}

\subsubsection{Proof of Theorem \ref{theo}}

Since the triple point case was covered in \cite{cat21},
we only have to deal with double point singularities.
The general statement that a normal quartic in characteristic 2
contains at most 14 singularities
was proved in \ref{ss:pf-thm} (using Propositions \ref{prop:>14} and \ref{prop:cusps}).
That 14 singular points necessarily form nodes was proved in Theorem \ref{thm:14nodes}.
An irreducible component was exhibited in Theorem \ref{dual=plane}.
Finally, the result that there are fewer than 14 singularities if the resolution is not a supersingular K3 surface,
was proved in Proposition \ref{prop:nu<14}.

\qed

\subsection*{Acknowledgement} 
Thanks  to the anonymous referees for their comments
which helped us improve the paper.
We would also like to thank Stephen Coughlan
for an interesting conversation.

\end{document}